\documentclass[a4paper]{article}

\usepackage{amsfonts}
\usepackage{amssymb}
\usepackage{graphicx}
\usepackage{amsmath}
\usepackage[a4paper]{geometry}

\newtheorem{theorem}{Theorem}

\newtheorem{corollary}{Corollary}

\newtheorem{lemma}{Lemma}

\newtheorem{proposition}{Proposition}
\newtheorem{assumptions}{Assumptions}

\newenvironment{proof}[1][Proof]{\noindent\textbf{#1.} }{\ \rule{0.5em}{0.5em}}
\def\lessim{\ \lower4pt\hbox{$\buildrel{\displaystyle <}\over\sim$}\ }

\begin{document}

\title{On the estimation of smooth densities by strict probability densities at optimal rates in sup-norm}
\author{\textsc{Evarist Gin\'{e}\footnote{This author is grateful for the hospitality of the MIT Mathematics Department, where he carried out most of his research on this article during a sabbatical leave.}   and Hailin Sang\footnote{Most of this author's research on this article was done  at  the Department of Mathematics of the University of Cincinnati, and he acknowledges its hospitality.}} \\
\\
\textit{University of Connecticut} and \textit{National Institute of Statistical Sciences}}
\date{May 2010}
\maketitle

\begin{abstract}
\noindent  It is shown that the variable bandwidth density estimator proposed by McKay (1993a and b)  following earlier findings by Abramson (1982) approximates  density functions in $C^4(\mathbb R^d)$ at the minimax rate in the supremum norm over bounded sets where the preliminary density estimates on which they are based are bounded away from zero. A somewhat more complicated estimator proposed by  Jones  McKay and Hu (1994) to approximate densities in $C^6(\mathbb R)$ is also shown to attain minimax rates in sup norm over the same kind of sets. These estimators are strict probability densities.  
\end{abstract}



\hfuzz=1truein

\section{Introduction and statement of results}
Let $X_i, i\in \mathbb N$, be independent identically distributed (i.i.d.) observations with density function $f(t)$, $t\in \mathbb R$ (to be replaced below by $t\in\mathbb R^d$). Setting $K$  to be a symmetric probability kernel satisfying some smoothness and differentiability properties, Abramson (1982) proposed the following `ideal' or `oracle' variable bandwidth kernel density estimator:
\begin{equation}\label{abramson}
f_A(t;h_n)=n^{-1}\sum_{i=1}^n h_n^{-1}\gamma (t, X_i)K(h_n^{-1}\gamma (t, X_i)(t-X_i)),
\end{equation}
where, $\gamma (t,s)=(f(s)\vee f(t)/10)^{1/2}$, which is made into a `real' estimator by replacing $f$ with a preliminary estimator. In words, in Abramson's estimator the window-width about each observation $X_i$ is inversely proportional to the square root of the density $f$ at $X_i$ unless $f(X_i)$ is too small, with the modification of $\gamma(t,X_i)$ for small values of $f(X_i)$  preventing against the possibility that observations $X_i$ very far away from $t$ exert too much influence on the estimate of $f(t)$. This estimator adapts to the local density of the data, and if the adaptation is adequate, which it is,  it seems that it should do better than the usual `fixed bandwidth' kernel density estimator: in fact Abramson showed that, while the variance of his estimator is of the same order as that of the regular kernel density estimator,  its bias is asymptotically of the order of $h_n^4$, assuming $f$ has four uniformly continuous derivatives and $f(t)\ne 0$ (while the 
  bias achieved by a symmetric non-negative kernel is of the order of only $h_n^2$). So, one has a {\it non-negative} estimator of the density that performs asymptotically as a kernel estimator based on a fourth order (hence, partly negative) kernel. However,  this variable bandwidth estimator is not the density function of a true probability measure since the integral of $f_A(t;h_n)$ over $t$ is not  $1$ -it would if $\gamma$ depended only on $s$-.  Terrell and Scott (1992) and  McKay (1993b) constructed different examples showing that the Abramson ideal estimator without the `clipping filter' $(f(t)/10)^{1/2}$ on $f^{1/2}(X_i)$,
$$f_{HM}(t;h_n)=n^{-1}\sum_{i=1}^n h_n^{-1}f^{1/2}(X_i)K(h_n^{-1}f^{1/2}(X_i)(t-X_i)),$$
which is a true probability density, may have bias of order much larger than $h_n^4$, and in fact their examples show that clipping is necessary for such bias reduction.   Hall, Hu and Marron (1995) then proposed the ideal estimator
\begin{equation}\label{HHM}
f_{HHM}(t;h_n)=\frac{1}{n h_n}\sum_{i=1}^{n}K\left(\frac{t-X_i}{h_n} f^{1/2}(X_i)\right) f^{1/2}(X_i)I(|t-X_{i}|<h_nB)
\end{equation}
where $B$ is a fixed constant; see also Novak (1999) for a similar estimator.
This estimator is non-negative and achieves the desired bias reduction but, like Abramson's, it does not integrate to 1. McKay (1993a and b) discovered a smooth clipping procedure which solves the problem of obtaining a non-negative ideal estimator that integrates to 1 and that has a bias of the order of $h_n^4$ for densities with four continuous derivatives. He used in (\ref{abramson}) a function $\gamma(t,s)=\gamma(s)$ not dependent on $t$, of the form  
\begin{equation}\label{clipf1}
\gamma(s):=\alpha(f(s)):=c\nu(\sqrt{f(s)}/c):=cp^{1/2}(f(s)/c^2),
\end{equation}
where the function $p$ (or the function $\nu$) is at least four times differentiable and satisfies $p(x)\ge 1$ for all $x$ and $p(x)=x$ for all $x\ge t_0$ for some $0<t_0<\infty$, and $0<c<\infty$ is a fixed number (note $p(x)=\nu^2(\sqrt{x})$ and while McKay (1993b)  uses $\nu$ we will use $p$ for convenience in calculations later). Functions $p$ with these properties will be denoted by {\it clipping functions}. Then, McKay's ideal estimator is
\begin{equation}\label{McKayideal}
f_{McK}(t;h_n)=n^{-1}\sum_{i=1}^n h_n^{-1}\alpha(f (X_i))K(h_n^{-1}\alpha(f (X_i))(t-X_i)).
\end{equation}
The bias reduction to $h_n^4$ is obtained uniformly over regions where $f(t)$ is bounded away from zero (or, if one allows $c$ in (\ref{clipf1}) to vary with $h_n$, uniformly in $f\in \mathbb R$).
Note that $\gamma$ may depend on $h_n$ as well and still have $\tilde{f}(t;h_n)$ integrate to 1. So, one may ask if with a more general function $\gamma(s,h)$ one can achieve further bias reduction.  McKay (1993b) and Jones, McKay and Hu (1994) show that using $\gamma(s,h)=\alpha(f(s))(1+h^2\beta(s))$, with $\alpha$ as above and a convenient function $\beta$ that depends on $f$, $f'$ and $f''$, a bias of the order of $h_n^6$ can be achieved on densities that are six times differentiable. This new estimator may be much less practical than the previous one since, in order to implement it, one has to obtain preliminary estimates not only of $f$ but also of its first two derivatives; moreover, these authors claim that preliminary simulations  with the ideal estimators show only modest gains by this new estimator over  (\ref{McKayideal}). 

The McKay and  Jones-McKay-Hu ideal estimators mentioned so far achieve bias reduction by adapting the bandwidth about each $X_i$ to the size of $f(X_i)$, with smooth clipping for small values of $f(X_i)$ and using kernels that are concentrated enough,  in order to keep the estimators local, and moreover they are strict probability densities. Samiuddin and El-Sayyad (1990) achieved the same results by shifting the centers of the windows by random quantities. See Jones, McKay and Hu (1994) who show that, by combining the two methods one can obtain an infinite number of such estimators, the general form of their ideal counterparts being
\begin{equation}\label{generalideal}
\bar{f}(t;h_n)=n^{-1}\sum_{i=1}^n h_n^{-1}\gamma (X_i)K(h_n^{-1}\gamma (X_i)(t-X_i-h_n^2\Gamma(X_i))),
\end{equation}
where the functions $\gamma$ and $\Gamma$ may depend on the bandwidth $h_n$, the density function $f(t)$ and its derivatives, and they considered $\gamma(z)=\alpha(f(z))(1+h_n^2\beta(z))$ and $\Gamma(z)=A(z)+h_n^2B(z)$ where $\alpha, ~\beta,~A$ and $B$ do not depend on $h_n$, but depend on $f$ (and $\beta$ and $B$ on its derivatives as well). However, Jones, McKay and Hu (1994) argue that among these, the most practical is McKay's modification of Abramson's estimator based on (\ref{McKayideal}), followed, at a distance,   by the one with $\Gamma=0$ and $\gamma(z)=\gamma(z,h)$ in (\ref{generalideal}) mentioned above, and we will pay attention only to these estimators in this article.

The estimators (\ref{abramson}) to (\ref{generalideal})  are usually called ideal estimators in the literature, and they give rise to true estimators $\hat{f}$ by replacement of  the density and its derivatives in their formulas by preliminary kernel estimators, perhaps using different sequences of bandwidths and different kernels, as in (\ref{realest}) and (\ref{h6idealest}) below.
The estimators $\hat f(t)$ are non-linear and it is difficult to measure their discrepancy from $f(t)$. After Hall and Marron (1988), this task is divided into two parts, a) the study of the ideal estimator, and b)
the study of the discrepancy between the ideal and the real estimators. 

The literature emphasizes the bias part of the ideal estimators, and on this one may say that the work of McKay (1993a and b) and Jones, McKay and Hu (1994) is final: the clipped estimators achieve bias reduction uniformly in regions where the density is bounded above from zero, and clipping is necessary for this reduction. Regarding the variance part of the ideal estimators, it is only shown in the literature that it is pointwise of the same order as the usual kernel density estimators, but there are no published  results on the uniform closeness of the (ideal) estimator to its mean except for one in Gin\'e and Sang (2010) for the estimator (\ref{HHM}) of Hall, Hu and Marron (1995). The discrepancy between the ideal and the corresponding real estimators for estimators based on Abramson's square root law turns out to be exactly of the same order as the difference between the ideal and the true density $f$, not less, and this discrepancy was first considered in detail by Hall and Marron (1988) and Hall, Hu and Marron (1995), who proved that it is asymptotically of the order of $n^{-4/9}$, {\it pointwise} and {\it in probability} for bounded densities with four bounded derivatives. McKay (1993b) adapted their method of proof and corrected some inaccuracies from Hall and Marron (1988) to show that this discrepancy for the multidimensional analogue of (\ref{McKayideal}) is of the order of $n^{-4/(8+d)}$, also pointwise and in probability, and for dimension $d<6$. Gin\'e and Sang (2010) show that, in the case of the Hall, Hu and Marron estimator and in dimension 1, the discrepancy is of the order of $((\log n)/n)^{4/9}$ almost surely and uniformly over intervals where the preliminary estimator is bounded away from zero,  as well as uniformly over densities with fixed but arbitrary bounds on their sup norm and the sup norms of their first four derivatives, thus obtaining a complete result on the uniform rate of approximation of the true density by  the real estimator corresponding to (\ref{HHM}) for $d=1$. Several of our arguments simplified by undersmoothing the preliminary estimator. These rates are optimal. In this article we prove  similar results, in $\mathbb R^d$ and without undersmoothing, for the McKay (1993b) estimator based on the generalization of  (\ref{McKayideal}) to $\mathbb R^d$, for any $d<\infty$, and also, but only in $\mathbb R$, for the estimator   with $\Gamma=0$ and $\gamma(z,h)=\alpha(f(z))/(1+h^2\beta(z))$ in (\ref{generalideal}) (these estimators can be handled in general and in $\mathbb R^d$, but the details are cumbersome and the results might not be too practical, according to the Jones-McKay-Hu (1994) study). In order to obtain these results we use empirical process theory, particularly and repeatedly, Talagrand's (1996) exponential inequality for empirical processes (which can often be replaced by a general result of Mason and Swanepoel (2010) that in addition yields uniformity in bandwidth) and, also at an important instance, an exponential inequality of Major (2006) for canonical $U$-processes, tools that were not available to previous authors, and that were introduced in density estimation respectively by Einmahl and Mason (2000) and Gin\'e and Mason (2007). We now describe our results.

Define
\begin{equation}\label{region0}
{\cal D}_r={\cal D}_r(f):=\{t\in\mathbb R^d: f(t)>r>t_0 c^2, \|t\|<1/r\},\ \ r>0.
\end{equation}
Here, $c$ and $t_0$ are the constants that appear in the clipping function $\gamma$ in (\ref{clipf1}).
In this paper, we obtain the optimal almost sure rate of uniform convergence  on the sets ${\cal D}_r$ for the two Jones-McKay-Hu estimators derived from Abramson's (1982) square root rule.  For the first, with bias $h_n^4=h_{2,n}^4$, the ideal estimator  is the multidimensional version of (\ref{McKayideal}),
\begin{equation}\label{idealest}
f_{McK}(t;h_{2,n})=\frac{1}{n h_{2,n}^d}\sum_{i=1}^{n}K\left(\frac{t-X_i}{h_{2,n}}\alpha( f(X_i))\right)\alpha^d(f(X_i))
\end{equation}
with 
\begin{equation}\label{alph}
\alpha(x)=cp^{1/2}(c^{-2}x), \ \ x\ge 0.
\end{equation} 
It is the special case of  (\ref{generalideal}) for $\Gamma=0$ and $\gamma(t)=\alpha(f(t))$ independent of $h_n$.  The estimator itself is
\begin{equation}\label{realest}
\hat f(t;h_{1,n},  h_{2,n})=\frac{1}{n h_{2,n}^d}\sum_{i=1}^{n}K\left(\frac{t-X_i}{h_{2,n}}\alpha(\hat f(X_i;h_{1,n}))\right)\alpha^d(\hat f(X_i;h_{1,n})),
\end{equation}
 where
$\hat f(x;h_{1,n})$ is the classical kernel density estimator
\begin{equation}\label{real}
\hat f(t;h_{1,n})=\frac{1}{n h_{1,n}^d}\sum_{i=1}^{n}K\left(\frac{t-X_i}{h_{1,n}}\right).
\end{equation}
The following notation will be convenient: ${\cal P}_C$ will denote the set of all probability densities on $\mathbb R^d$
that are uniformly continuous and are bounded by $C<\infty$, and ${\cal P}_{C,k}$ will denote the set of densities on $\mathbb R^d$ for which themselves and their partial derivatives of  order $k$ or lower are bounded by $C<\infty$ and are uniformly continuous. The dependence on the dimension $d$ will be left implicit both for the regions ${\cal D}_r$ and for the sets of densities ${\cal P}_{C,k}$.

Here is our first theorem:

\begin{theorem}\label{main0}
Assume that the kernel $K$ on $\mathbb R^d$ is non-negative, integrates to 1 and has the form $K(t)=\Phi(\|t\|^2)$  for some real twice boundedly differentiable even function $\Phi$  with support contained in $[-T,T]$, $T<\infty$. Let $\alpha (x)$ by defined by (\ref{alph}) for a nondecreasing clipping function $p(s)$ ($p(s)\ge 1$ for all $s$ and $p(s)=s$ for all $s\ge t_0\ge 1$) with five bounded and uniformly continuous derivatives, and constant $c>0$.   Set $h_{2,n}= ((\log n)/n)^{1/(8+d)}$ and $h_{1,n}=((\log n)/n)^{1/(4+d)}$,  $n\in\mathbb N$. Then, the estimator $\hat f(t;h_{1,n},  h_{2,n})$ given by (\ref{realest}) and (\ref{real}) with the kernel, bandwidths and function $\alpha$ just described, satisfies
\begin{equation}\label{main1}
\sup_{t\in {\cal D}_r(f)}\left|\hat f(t;h_{1,n},  h_{2,n})-f(t)\right|=O_{\rm a.s.}\left(\left(\frac{\log n}{n}\right)^{4/(8+d)}\right)\ \ uniformly\ \ in\ \ f\in{\cal P}_{C,4}
\end{equation}
for any $C<\infty$.
If $\hat {\cal D}_r^n(f)$ is defined as
\begin{equation}\label{region1}
\hat D_r^n(f)=\left\{t: \hat f(t;h_{1,n})>2r>t_0 c^2, \|t\|<1/r\right\},
\end{equation}
then we also have
\begin{equation}\label{main2}
\sup_{t\in \hat {\cal D}_r^n}\left|\hat f(t;h_{1,n},  h_{2,n})-f(t)\right|=O_{\rm a.s.}\left(\left(\frac{\log n}{n}\right)^{4/(8+d)}\right)\ \ uniformly\ \ in\ \ f\in{\cal P}_{C,4}.
\end{equation}
\end {theorem}

We should recall that, given measurable functions $Z_{n,f}(X_1,\dots, X_n)$, $X_i$ the coordinate functions of $({\mathbb R^d})^{\mathbb N}$, $f\in {\cal D}$, $\cal D$ a collection of densities on $\mathbb R^d$,  we say that the collection of random variables $Z_{n,f}(X_1,\dots, X_n)$ is asymptotically of the order of $a_n$ {\it uniformly in $f\in\cal D$} if there exists $C<\infty$ such that
$$\lim_{k\to\infty}\sup_{f\in{\cal D}}(P_f)^{\mathbb N}\left\{\sup_{n\ge k}\frac{1}{a_n}|Z_{n,f}(X_1,\dots,X_n)|>C\right\}=0,$$
where $dP_f(x)=f(x)dx$, and that it is $o_{\rm a.s.}(a_n)$ uniformly in $f$ if this limit holds for every $C>0$. In the text we will use $\Pr_f$ for $(P_f)^{\mathbb N}$, or even $\Pr$ if $f$ is understood from the context.

We should note that whereas the bias of the ideal estimator is of the right order uniformly only over ${\cal D}_r$, both the variance part of the ideal estimator and the difference between ideal and real estimators are of the right order uniformly in $\mathbb R^d$.

Here is an example of a five times differentiable clipping function $p$ for which $t_0=2$:
\begin{displaymath}
p(t)=\left\{\begin{array}{ll}
1+\frac{t^6}{64}\left(1-2(t-2)+\frac{9}{4}(t-2)^2-\frac{7}{4}(t-2)^3+\frac{7}{8}(t-2)^4\right) & \textrm{if $0\le t\le2$}\\
t&\textrm{if $t\ge 2$}\\
1&\textrm {if $t\le 0$}\end{array}
\right..
\end{displaymath}
(We could as well take this formula as the definition of $\nu$, and set $p(x)=\nu(\sqrt{x})$ as in (\ref{clipf1}).) This is based on McKay's (1993b) example of a four times differentiable clipping function. Other examples of such functions are possible, and in particular see McKay, loc. cit., for an infinitely differentiable one.

In $\mathbb R$, the ideal estimator with bias $h_n^6=h_{2,n}^6$ that we will consider has the  form
\begin{equation}\label{h6idealest}
f_{JKH}(t;h_{2,n})=\frac{1}{n h_{2,n}}\sum_{i=1}^{n}K\left(\frac{t-X_i}{h_{2,n}}\gamma_{h_{2,n}}(X_i)\right)\gamma_{h_{2,n}}(X_i),
\end{equation}
where
\begin{equation}\label{alphabeta}
\gamma_{h_{2,n}}(x)=\frac{\alpha(f(x))}{1+h_{2,n}^2\beta(x)}\  {\rm with}\ \alpha(f(x))=cp^{1/2}(c^{-2}f(x)), \   \beta(x)=\frac{\tau_4[f^{\prime\prime}(x)f(x)-2(f^\prime(x))^2]}{24\tau_2\alpha^6(f(x))}
\end{equation}
with $$\tau_r=\int K(x)|x|^rdr,\ \ r>0,$$
((\ref{h6idealest}) is another special case of the general estimator (\ref{generalideal}) with $\Gamma=0$, but with $\gamma$ depending on the bandwidth $h_{2,n}$).
The true estimator corresponding to (\ref{h6idealest}) is
\begin{eqnarray}\label{h6realest}
&&\hat f(t;h_{1,n},h_{2,n},h_{3,n},h_{4,n})\notag\\
&&=\frac{1}{n h_{2,n}}\sum_{i=1}^{n}K\left(\frac{t-X_i}{h_{2,n}}\hat \gamma(X_i;h_{1,n},h_{2,n},h_{3,n},h_{4,n})\right)\hat \gamma(X_i;h_{1,n},h_{2,n},h_{3,n},h_{4,n}),
\end{eqnarray}
where
$$\hat \gamma(x;h_{1,n},h_{2,n}, h_{3,n},h_{4,n})=\frac{\hat\alpha(x;h_{1,n})}{1+h_{2,n}^2\hat\beta(x;h_{1,n},h_{3,n},h_{4,n})},\ \hat\alpha(x;h_{1,n}):=\alpha(\hat f(x; h_{1,n})),$$ 
\begin{equation}\label{alphabetahat}
\hat\beta(x;h_{1,n},h_{3,n},h_{4,n})=\frac{\tau_4[f_{G_2}(x;h_{4,n})\hat f(x; h_{1,n})-2(f_{G_1}(x;h_{3,n}))^2]}{24\tau_2\hat\alpha^6(x;h_{1,n})}.
\end{equation}
 Here $\hat f(x;h_{1,n})$ is the classical kernel density estimator (\ref{real}), and $f_{G_1}$ and $f_{G_2}$ are the estimators of $f'$ and $f''$ given respectively by
\begin{equation}\label{first}
f_{G_1}(x;h_{3,n})=\frac{1}{n h_{3,n}^2}\sum_{i=1}^{n}G^\prime\left(\frac{x-X_i}{h_{3,n}}\right)
\end{equation}
and
\begin{equation}\label{second}
f_{G_2}(x;h_{4,n})=\frac{1}{n h_{4,n}^3}\sum_{i=1}^{n}G^{\prime\prime}\left(\frac{x-X_i}{h_{4,n}}\right),
\end{equation}
where $G$ is a fourth order kernel, that is, such that
\begin{displaymath}
\int z^j G(z)dz= \left\{\begin{array}{ll} 1 & \textrm{if $j=0$}\\0 & \textrm{if $1\leqslant j \leqslant 3$}\\ a\neq 0 & \textrm{if $j=4$} \end{array}\right.
 \end{displaymath}
We will sketch the proof of  the following theorem for this estimator:

\begin{theorem}\label{h6main0}
Assume the kernel $K$ is as in Theorem \ref{main0}.  Assume that the fourth order kernel $G$ is supported by $[-T_G,T_G]$ for some $T_G<\infty$, is twice continuously differentiable,  is symmetric about zero and integrates to 1. Let $\alpha (x)$ by defined by (\ref{alph}) for a nondecreasing clipping function $p(s)$ ($p(s)\ge 1$ for all $s$ and $p(s)=s$ for all $s\ge t_0\ge 1$) with seven bounded and uniformly continuous derivatives, and constant $c>0$, and let $\gamma$ and $\beta$ be as in (\ref{alphabeta}). Set $h_{1,n}=((\log n)/n)^{1/5}$, $h_{2,n}=h_{4,n}= ((\log n)/n)^{1/13}$ and $h_{3,n}=((\log n)/n)^{1/11}$,  $n\in\mathbb N$. Let $\hat \gamma$, $\hat\alpha$, $\hat \beta$ be as in (\ref{alphabetahat}) for these bandwidths, and let $\hat f$ be defined by (\ref{h6realest}) with the kernels, bandwidths and function $\hat\gamma$ just described. Then we have
\begin{equation}\label{h6main1}
\sup_{t\in D_r}\left|\hat f(t;h_{1,n},  h_{2,n}, h_{3,n},h_{4,n})-f(t)\right|=O_{\rm a.s.}\left(\left(\frac{\log n}{n}\right)^{6/13}\right)\ \ uniformly\ \ in\ \ f\in{\cal P}_{C,6}.
\end{equation}
Further,  for the region defined in (\ref{region1}), we have
\begin{equation}\label{h6main2}
\sup_{t\in \hat D_r^n}\left|\hat f(t;h_{1,n},  h_{2,n}, h_{3,n},h_{4,n})-f(t)\right|=O_{\rm a.s.}\left(\left(\frac{\log n}{n}\right)^{6/13}\right)\ \ uniformly\ \ in\ \ f\in{\cal P}_{C,6}.
\end{equation}
\end {theorem}

It seems natural that Theorem  \ref{h6main0} extends to several dimensions, like Theorem \ref{main0}, but  the proof is already quite involved for $d=1$ and the practical value of this estimator is not proven (cf. Jones, McKay and Hu (1994)).

\section{Bias and variance of the ideal estimator}

In this section we consider the ideal estimators $f_{McK}$ and $f_{KJH}$ in several dimensions. We a) describe  the bias reduction for the ideal estimators, mainly following  McKay (1993a and b) (see also Jones,  McKay and Hu (1994)), and b) show that the uniform rate of concentration of the ideal estimators about their means, not surprisingly, turn out to be the same as those of regular kernel density estimators in $\mathbb R^d$ (Gin\'e and Guillou (2002); Deheuvels (2000) in one dimension). 

\subsection{Uniform bias expansions}

Our ideal estimator is
\begin{equation}\label{locationideal}
\bar{f}(t;h_n)=\bar f_n(t;h_n)=n^{-1}\sum_{i=1}^n h_n^{-d}\gamma_h^d(X_i)K(h_n^{-1}\gamma_h (X_i)(t-X_i)),\ \ t\in\mathbb R^d
\end{equation}
that is, we take $\Gamma=0$ in (\ref{generalideal}). Since $\gamma=\gamma_h$ may depend on $h$, in order to handle this carefully it is better to assume that $\gamma$ depends on another variable $\delta$ and eventually have $\delta=h$:
\begin{equation}\label{locationideal2}
\bar f_n(t;h,\delta)=n^{-1}\sum_{i=1}^n h^{-d}\gamma^d_\delta (X_i)K(h^{-1}\gamma_\delta (X_i)(t-X_i)),\ \ t\in\mathbb R^d.
\end{equation}
The general case (\ref{generalideal}) can be considered in exactly the same way, but formulas become more complicated.

The following proposition and its proof are contained in McKay (1993b, Theorems 2.10,  1.1 and 5.13) (see also Hall (1990), and particularly Jones,  McKay and Hu (1994, Theorem A.1) and McKay (1993b)). We sketch these authors' proof in the case $d=1$  for the reader's convenience.

Notation: we say that a function $g$ is in $C^l(\Omega)$ if itself and its first $l$ derivatives are bounded and uniformly continuous on $\Omega$.

More notation: for $v=(v_1,\dots,v_d)\in \mathbb N\cup\{0\})^d$, we set $|v|=\sum_{i=1}^dv_i$, $D_v:=D_{x_1}^{v_1}\circ\dots\circ D_{x_d}^{v_d}$,  $v!=v_1!\cdots v_d!$ and $\tau_v=\int_{\mathbb R^d} u_1^{v_1}\cdots u_d^{v_d}K(u)du$.

\begin{proposition} \label{bias0} {\rm(McKay (1993a,b))}
Let the kernel $K:\mathbb R^d\mapsto \mathbb R$ be symmetric about zero separately in each coordinate, have bounded support and integrate to 1. Assume the density $f$ is in $C^l(\mathbb R^d)$.  Assume $\gamma_\delta(t)\ge c>0$ for some $c>0$ and all $t\in \mathbb R^d$ and  $0\le \delta\le\delta_0$, for some $\delta_0>0$, and that the function $\gamma(t,\delta):=\gamma_\delta(t)$ is in $C^{l+1}(\mathbb R^d\times [0,\delta_0])$. Then we have
\begin{equation}\label{locationideal3}
E\bar f_n(t;h,\delta)=\sum_{k=0}^l a_{k,\delta}(t)h^k+o(h^l)
\end{equation}
as $h\to 0$, uniformly in $t\in\mathbb R^d$ and $0\le\delta\le\delta_1$ for some $\delta_1>0$, and the set of functions $a_{k,\delta}$, which are uniformly bounded and equicontinuous, are defined as 
\begin{equation}\label{coeff}
a_{2k+1,\delta}(t)=0,\ \ a_{2k,\delta}(t)=\sum_{|v|=2k}\frac{\tau_{v}}{v!}D_v\left(\frac {f(t)}{\gamma_{\delta}^{2k}(t)}\right),
\end{equation}
for $k\le l/2$, in particular, $a_{0,\delta}(t)=f(t)$.
\end{proposition}

\begin{proof} (For $d=1$.) The difference between the proof of this proposition in dimension 1 and in dimension $d>1$ is that in dimension 1 we can use positivity of the derivatives of a certain function in order to partially invert it, whereas in the case of $\mathbb R^d$ we need to use  the implicit function theorem. We refer to Lemma 2.11 in McKay (1993b) for the details in any dimensions, but as mentioned above, we only consider here the case $d=1$.
Since the functions $\gamma_\delta$ are bounded away from zero and their derivatives are bounded, there exists $\delta_1>0$ such that $\gamma_\delta(t-v)-v\gamma'_\delta(t-v)$ is bounded away from zero for all $t\in \mathbb R$, $\delta\in[0,\delta_0]$, and $v\in [-\delta_1,\delta_1]$. Hence,
for each $t\in\mathbb R$ and $0\le\delta\le \delta_0$, the function $v\mapsto U_{t,\delta}(v):=v\gamma_\delta(t-v)$ is invertible on the neighborhood $[-
\delta_1,\delta_1]$ of $v=0$.  These inverse functions, say $V_{t,\delta}(u)$,  are $l+1$ times differentiable with continuous derivatives, with respect to the three variables (this can be seen directly by differentiation, or using the implicit function theorem as in McKay (1993b) Theorem 2.10 and Lemma 2.11 for $ {\bf x}=(t,\delta)$). If the support of $K$ is $[-T,T]$ then $K(h^{-1}\gamma_\delta (s)(t-s))=0$ unless $|t-s|\le hT/c$. This implies that the change of variables
$$hz=(t-s)\gamma_\delta(t-(t-s)),\ {\rm that\ is}\ t-s=V_{t,\delta}(hz),$$ in the following integral is valid for all $h$ small enough
\begin{eqnarray*}E\bar f(t;h,\delta)&=&\frac{1}{h}\int\gamma_\delta(s)f(s)K\left(\frac{t-s}{h}\gamma_\delta(s)\right)ds\\
&=&-\int
\gamma_\delta(t-V_{t,\delta}(hz))f(t-V_{t,\delta}(hz))\frac{dV_{t,\delta}(hz)}{d(hz)}K(z)dz.
\end{eqnarray*}
 Now, the first statement in the proposition follows by developing the function $\gamma_\delta(t-V_{t,\delta}(hz))f(t-V_{t,\delta}(hz))\frac{dV_{t,\delta}(hz)}{d(hz)}$ into powers of $hz$ and integrating, on account of the compactness of the domain of integration ($z\in[-T,T]$) and the differentiability properties of $f$ and $\gamma_\delta$. (Note that the presence of $dV(hz)/d(hz)$ in the integrand requires that  the function $V$ be $l+1$ times differentiable in order to obtain differentiability of the integrand up to the $l$-th order, necessary for (\ref{locationideal3}).)

 Let $\psi$ be an infinitely differentiable function of bounded support. Then, changing variables ($t=s+hu$), developing $\psi$, changing variables once more ($w=u\gamma_\delta(s)$)) and  integrating by parts, we obtain
 \begin{eqnarray*}
 \int \psi(t)E\bar f(t;h,\delta)dt&=&\int\int \psi(s+hu)\gamma_\delta(s)f(s)K(u\gamma_\delta(s))dsdu\\
 &=&\sum_{k=0}^l\frac{h^k}{k!}\int\psi^{(k)}(s)\gamma_\delta(s)f(s)\left(\int u^kK(u\gamma_\delta(s))du\right)ds
 +o(h^l)\\
 &=&\sum_{k=0}^l\frac{\tau_kh^k}{k!}\int \psi^{(k)}(s)\frac{f(s)}{\gamma_\delta^k(s)}ds+o(h^l)\\
 &=&\int\psi(s)f(s)ds+\sum_{k=1}^l(-1)^k\frac{\tau_kh^k}{k!}\int \psi(s)\left(\frac{f(s)}{\gamma_\delta^k(s)}\right)^{(k)}ds+o(h^l),
 \end{eqnarray*}
 and note that, by symmetry, $\tau_k=0$ if $k$ is odd.
 But by (\ref{locationideal3}), $$ \int \psi(t)E\bar f(t;h,\delta)dt=\sum_{k=0}^lh^k\int \psi(t)a_{k,\delta}(t)dt+o(h^l),$$ and (\ref{coeff}) follows by comparing the coefficients of $h^l$ in both expansions.
 \end{proof}

 \smallskip

 With a slightly less simple proof, one can replace the bounded support hypothesis on $K$ by $\int(1+|x|^l)K(x)dx<\infty$, as done in the above mentioned references. 
\smallskip

\begin{corollary}\label{4}
{\rm (McKay (1993a,b))} Let $f$ be a  density in $C^4(\mathbb R^d)$, let $p$ be a  clipping function in $C^5(\mathbb R)$,  set $\alpha(f(t))=cp^{1/2}(c^{-2}f(t))$ for some $c>0$, and define $\bar f(t,h)$ by equation (\ref{locationideal}) with $\gamma(s)=\alpha(f(s))$, that is  $\bar f(t;h)=f_{McK}(t;h)$ (see (\ref{idealest})). Let ${\cal D}_r$ be as in
(\ref{region0}). Then,
\begin{equation}\label{idealbias4}
Ef_{McK}(t;h)=f(t)+\left(\sum_{|v|=4}\tau_vD_v(1/f)/v!\right)h^4
+o(h^4)
=f(t)+O\left(h^4\right)
\end{equation}
as $h\to0$, uniformly on ${\cal D}_r$.
\end{corollary}

\begin{proof}
For $x\in {\cal D}_r$, $\gamma(t)=cp^{1/2}(c^{-2}f(t))=f^{1/2}(t)$, so that, by equation (\ref{coeff}), $a_2(x)=0$ on ${\cal D}_r$. So, the corollary follows from the previous proposition.
\end{proof}

\begin{corollary}\label{6}
{\rm (McKay (1993a), Jones, McKay and Hu (1994))} Let $f$ be a density in $C^6(\mathbb R)$.  Let $p$ be a clipping function in $C^7(\mathbb R)$ and, for some $c>0$,  set $\alpha(f(t))=cp^{1/2}(c^{-2}f(t))$ and
$$\beta(t)=\frac{\tau_4[f^{\prime\prime}(t)f(t)-2(f^\prime(t))^2]}{24\tau_2\alpha^6(t)}.$$
Define
\begin{equation}\label{gdelta}
\gamma_\delta (t)=\frac{\alpha(t)}{1+\delta^2\beta(t)},
\end{equation}
and, consider $\bar f_n(t;h,h)$, the estimator defined by (\ref{locationideal2}) with this $\gamma_\delta$ and with $\delta=h$, that is $\bar f_n(t;h,h)=f_{JKH}(t;h)$ (see \ref{h6realest}).  Let ${\cal D}_r$ be as in
(\ref{region0}) for dimension $d=1$. Then,
\begin{eqnarray}\label{h6''}
Ef_{JKH}(t;h)&=&f(t)+h^6\left\{\frac{1}{2}\tau_2(\beta^2)^{\prime\prime}(t)+\frac{1}{6}\tau_4\left(\frac{\beta}{f}\right)^{(4)}(t)
+\frac{1}{720}\tau_6\left(\frac{1}{f^2}\right)^{(6)}(t)\right\}+o(h^6)\nonumber\\
&=&f(t)+O\left(h^6\right)
\end{eqnarray}
as $h\to0$, uniformly on ${\cal D}_r$.
\end{corollary}

\begin{proof}  By definition,
$\alpha(f(t))=f^{1/2}(t)$ and $\beta(t)=\frac{\tau_4[f^{\prime\prime}(t)f(t)-2(f^\prime(t))^2]}{24\tau_2f^3(t)}=-\frac{\tau_4}{24\tau_2}\left(\frac{1}{f}\right)^{\prime\prime}(t)$ on ${\cal D}_r$, and the corollary follows by direct application of the previous proposition.
\end{proof}

\smallskip
The estimator for $l=6$ here is similar but slightly different from the one in Section 4.1 of Jones, McKay and Hu (1994)  (they defined $\gamma_\delta=\alpha(1+h^2\beta)$, but, it is somewhat more convenient for us to define it  by equation (\ref{gdelta})).

\subsection{Rate of uniform deviation from the mean}

The next proposition and its first proof have its origin in a result of Gin\'e and Guillou (2002) that obtains the almost sure exact discrepancy rate between kernel density estimators in $\mathbb R^d$ and their expected values, uniformly  on the whole space. A  proposition closer to the one below was proved in Gin\'e and Sang (2010) for the Hall-Hu-Marron estimator, and Mason and Swanapoel  (2010) (see also Mason (2010)) proved a general theorem that also yields the result, even with uniformity in bandwidth. Whereas the results in these references all imply (by exact analogy or as a direct consequence) the proposition below, we should emphasize that its proof consists of nothing but a direct and straightforward application of the famous Talagrand's exponential inequality, in the version in Einmahl and Mason (2000, inequality A.1 combined with Proposition A.1), and in 
Gin\'e and Guillou (2001, Proposition 2.2; 2002, Corollary 2.2), which turns out to be as well the main component in the proofs of all the above mentioned results. 

\begin{assumptions} \label{ass1}
The sequence $h_n$  has the form
\begin{equation}\label{band}
h_{n}=((\log n)/n)^\eta
\end{equation}
for some $0<\eta<1$. The kernel $K$ has the form $K(t)=\Phi(\|t\|^2)$, where $\Phi$ is bounded,  has support on $[0,T]$ for some $T<\infty$ and is of bounded variation and left or right continuous. $f$ is a bounded density function, and $\gamma_h(x)=\alpha(x)/(1+h^2\beta(x))$, where the functions $\alpha$ and $\beta$ are continuous and  bounded, $\alpha$ is bounded away from zero, and $0<h\le 1/(2\|\beta\|_\infty)^{1/2}$. The ideal estimator $\bar f(t;h_n)=\bar f_n(t)$ is defined, with these kernel, function $\gamma_h$ and bandwidths $h_n$, as in (\ref{locationideal}). 
\end{assumptions}

All these assumptions can be weakened, but this is all we need in this article.

We recall that $N({\cal K}, d,\varepsilon)$, the $\varepsilon$ covering number of the metric or pseudo-metric space $({\cal K},d)$,  is defined as the smallest number of  (open) $d$-balls of radius not exceeding $\varepsilon$ needed to cover $\cal K$. Also, a collection of measurable functions $\cal K$ on a measurable space $(S,{\cal S})$ is of $VC$ type relative to an envelope $F$ (a measurable function $F$ such that $F(s)\ge |f(s)|$ for all $s\in S$ and $f\in{\cal K}$) if there exist finite constants $A$, $v$ such that, for all probability measures $Q$ on $(S,{\cal S})$,
\begin{equation}\label{vapcer}
N({\cal K}, L_2(Q),\varepsilon)\le\left(\frac{A\|F\|_{L_2(Q)}}{\varepsilon}\right)^v, \ \ 0<\varepsilon\le 2\sup_{f,g\in{\cal K}}\|f-g\|_{L_2(Q)}.
\end{equation}
All but one among the classes of functions that we will consider in this article can be shown to be of $VC$ type using the following lemma, which is  a  variation on Lemma 4 of Gin\'e and Sang (2010), and whose idea comes from Nolan and Pollard (1987) (inexcusably, we failed to mention this article in Gin\'e and Sang (2010)).

\begin{lemma}\label{vc1}
 Let $K$, $f$  and $\gamma_h$  satisfy Assumptions \ref{ass1}. Let $\cal G$ be a uniformly bounded $VC$ type class of measurable functions on $\mathbb R^d$ with respect to a constant envelope $G$ and admitting constants $A_1$, $v_1$ in equation (\ref{vapcer}).
Let $\cal K$ be the class of functions
 \begin{equation}\label{classk0}
{\cal K}= \left\{K\left(\frac{t-\cdot}{h}\gamma_h(\cdot)\right)g(\cdot):t\in\mathbb{R}^d,~ 0<h<1/(2\|\beta\|_\infty)^{1/2}, ~g\in{\cal G}\right\}.
\end{equation}
Then, there exists a universal constant $R$ such that for every Borel probability measure $Q$ on $\mathbb R^d$,
\begin{equation}\label{ivc1}
N({\cal K},L_2(Q),\varepsilon)\le\left(\frac{(R\vee A_1)\|\Phi\|_VG}{\varepsilon}\right)^{8d+20+v_1}
\end{equation}
where $\|\Phi\|_V$ is the total variation norm of $\Phi$, that is, ${\cal K}$ is a bounded class of functions of VC type with envelope $\|\Phi\|_VG$ and admitting characteristic constants $A=R\vee A_1$ and $v=8d+20+v_1$,  independent of  $f$.
\end{lemma}

\begin{proof}
By adding an arbitrarily small strictly increasing function to the positive and negative variation functions of $\Phi$, we have $\Phi=\Phi_1-\Phi_2$ with $\Phi_i$ strictly increasing, positive and bounded, with $\|\Phi_1\|_\infty$ ($\|\Phi_2\|_\infty$) arbitrarily close to the positive (negative) variation of $\Phi$. For $i=1, 2$, let ${\cal K}_i$ be the classes of functions
$${\cal K}_i=\left\{\Phi_i\left(\frac{\|t-\cdot\|^2}{h^2}\gamma_h^2(\cdot)\right):t\in\mathbb{R}^d, 0<h<\infty\right\}.$$
Then the subgraphs of the functions in the class ${\cal K}_i$ have the form
\begin{equation*}
\left\{(x,u):\Phi_i\left(\frac{\|t-x\|^2}{h^2}\gamma_h^2(x)\right)\ge u \right\}
=\left\{(x,u):\frac{\|t-x\|^2}{h^2}\alpha^2(x)-\Phi_i^{-1}(u)(1+h^2\beta(x))^2
\ge0\right\},
\end{equation*}
and so they are the positivity sets of  functions  from the linear space of functions of the two variables $u$ and $x$ spanned by $\alpha^2(x)$, $x_j\alpha^2(x)$ for $j=1,\dots,d$,  $x_j^2\alpha^2(x)$ for $j=1,\dots,d$, $\Phi_i^{-1}(u)$, $\Phi_i^{-1}(u)\beta(x)$ and $\Phi_i^{-1}(u)\beta^2(x)$.
 Hence, by a result of Dudley (1978), the subgraphs of ${\cal K}_i$ are VC of index $2d+5$.
Therefore, by the Dudley-Pollard entropy bound for VC-subgraph classes, e.g.  Pollard (1984), we have
\begin{equation}\label{vc2}
N({\cal K}_i, L_2(Q), \varepsilon)\le \left(\frac{R\|\Phi_i\|_\infty}{\varepsilon}\right)^{4d+10},\ \ 0<\varepsilon\le \|\Phi_i\|_\infty,\ \ i=1,2,
\end{equation}
where $R$ is a universal constant. 

Now, any $H\in{\cal K}$ can be written as $H=k_1g-k_2g$ for $k_i\in{\cal K}_i$ and $g\in{\cal G}$, so that, for any probability measure $Q$ we have
\begin{eqnarray*}Q(H-\bar H)^2&=&Q((k_1-k_2)g-(\bar k_1-\bar k_2)\bar g)^2\\
&\le&4G^2Q(k_1-\bar k_1)^2+4G^2Q(k_2-\bar k_2)^2+2\|K\|_V^2Q(g-\bar g)^2.
\end{eqnarray*}
Given $\varepsilon>0$ let $\delta_1=\varepsilon/(\sqrt{12}G)$ and $\delta_2=\varepsilon/(\sqrt{6}\|K\|_V)$.
Then, if the collections of functions $k_1^{(1)},\dots, k_{N_1}^{(1)}$ and $k_1^{(2)},\dots, k_{N_2}^{(2)}$ are $L_2(Q)$ $\delta_1$-dense respectively in the classes ${\cal K}_1$, ${\cal K}_2$, and $g_1,\dots,g_{N_3}$ are $L_2(Q)$ $\delta_2$-dense  in the class $\cal J$, with optimal cardinalities $N_i=N({\cal K}_i, L_2(Q),\delta_1)$, $i=1,2$, and $N_3=N({\cal G},L_2(Q),\delta_2)$, then, by the previous inequality, the functions $(k_i^{(1)}-k_j^{(2)})g_l$ are $L_2(Q)$ $\varepsilon$-dense in $\cal F$ . Since there are at most $N_1N_2N_3$ such functions (this estimate may not be optimal), the inequality  (\ref{ivc1}) follows.
\end{proof}

\medskip

This lemma is important for us because it allows direct application of  the theorem in Mason and Swanepoel (2010)) or, what is more natural in our case, direct use of a version of Talagrand's (1996) inequality e.g. in the form given in Einmahl and Mason (2000) or in Gin\'e and Guillou (2001, 2002),  to the effect that, if $P$ is a probability measure on a measurable space $(S,{\cal S})$ and  $X_i:S^{\mathbb N}\mapsto S$ are the coordinate functions of $S^{\mathbb N}$, which are i.i.d. with law $P$, and if
 a class $\cal F$ of functions is bounded, countable and of VC type for an envelope $F$,  then there exist $0<C_i<\infty$, $1\le i\le 2$, depending only on $v$ and $A$ such that, for all $\lambda \ge 1\vee 2C_1$ and all $t$ satisfying
 \begin{equation} \label{talc}
 C_1\sqrt{n}\sigma\sqrt{\log\frac{2\|F\|_\infty}{\sigma}}\le t\le \frac{\lambda n\sigma^2}{\|F\|_\infty},
 \end {equation}
 we have
 \begin{equation}\label{tal}
\Pr\left\{\sup_{g\in{\cal F}}\left|\sum_{i=1}^n(g(X_i)-Pg)\right|>t\right\}\le C_2\exp\left(-\frac{t^2}{C_2\lambda n\sigma^2}\right),
\end{equation}
where
$$\|F\|_\infty\ge \sigma^2\ge \sup_{g\in {\cal F}}{\rm Var}_P(g).$$
The class $\cal K$ is not countable, but the continuity properties of the functions defining it imply that the sup over $g\in {\cal K}$ of $\left|\sum_{i=1}^k(g(X_i)-Pg)\right|$ is in fact a countable supremum. Whenever this will happen in this article we will say that the class is {\it measurable}.

\begin{proposition}\label{varid}
Under the hypotheses in Assumptions \ref{ass1},
\[
\sup_{t\in \mathbb{R}^d}|\bar{f}(t;h_n)-E\bar{f}(t;h_n)|=||\bar{f}_n-E\bar{f}_n||_{\infty}=O_{\rm a.s.}\left(\sqrt{\frac{\log n}{n h_{n}^d}}\right)
\]
uniformly over all densities $f$ such that $\|f\|_\infty\le C$, for any $0<C<\infty$, that is, there exists $L<\infty$ such that, if ${\cal P}_C$ is the set of these densities, then
$$\lim_{k\to0}\sup_{f\in{\cal P}_C}{\Pr}_f\left\{\sup_{n\ge k}\sqrt{\frac{n h_{n}^d}{\log n}} ||\bar{f}_n-E\bar{f}_n||_{\infty}>L\right\}=0.$$

\end{proposition}

\begin{proof} We have:
\begin{eqnarray}\label{eq1}
\Pr\left\{\sqrt{\frac{n h_{n}^d}{\log h_{n}^{-1}}}
||\bar{f}_n-E\bar{f}_n||_{\infty}>\lambda\right\} 
&=&\Pr\Biggr\{
\sup_{ {t\in \mathbb{R}^d} }\left|\sum_{i=1}^{n}\left[
K\left(\frac{t-X_i}{h}\gamma_h(X_i)\right)\gamma_h^d(X_i)
\right.\right.{}\\
&&\left.\left.{} -EK\left(\frac{t-X_i}{h}\gamma_h(X_i)\right)\gamma_h^d(X_i)\right]\right|
>\lambda\sqrt{n h_{n}^d\log n}\Biggr\}\notag
\end{eqnarray}
for any $\lambda>0$. By Lemma \ref{vc1}, for $0<h\le 1/(2\|\beta\|_\infty)^{1/2}$, the subclasses
\begin{equation}
\mathcal{F}_{h}=\left\{K\left(\frac{t-\cdot}h\gamma_h(\cdot)\right)\gamma_h^d(\cdot):t\in \mathbb{R}^d \right\} \label{entr0}
\end{equation}
are measurable VC classes of functions with respect to the  envelope $U=2^d\|K\|_V\|\alpha\|_\infty^d$, and admitting constants $A$ and $v$ independent of $h$: notice that the class of functions ${\cal G}=\{\gamma_h^d: 0<h\le 1/(2\|\beta\|_\infty)^{1/2}\}$, is bounded by $2^d\|\alpha\|_\infty^d$ and is  is clearly of VC type with $v=1$ since $\left|\left(\frac{\alpha(x)}{1+h^2\beta(x)}\right)^d-\left(\frac{\alpha(x)}{1+\bar h^2\beta(x)}\right)^d\right|\le d2^{d+1}\|\alpha\|_\infty^d(\|\beta\|_\infty)^{1/2}|h-\bar h|$; hence, the class of functions  $\cal K$ defined as in (\ref{classk0}) using  this $\cal G$ and the kernel $K$ and the functions $\gamma_h$ from this proposition, is 
$VC$ by Lemma \ref{vc1} and, since it contains  the classes ${\cal F}_h$, so are these classes, with the same $A$ and $v$ as $\cal K$. The continuity properties of $K$ and $\gamma_h$ imply they are measurable. Hence,
Talagrand's inequality (\ref{tal}) applies to the classes (\ref{entr0}), and in order to apply it 
 we only need a sensible bound $\sigma^2_h$ for the maximum variance of the functions in each class ${\cal F}_h$.  Since $0<\nu:=(2/3)\inf_x\alpha(x)\le \|\gamma_h\|_\infty\le 2\|\alpha\|_\infty<\infty$, we have
 \begin{eqnarray}\label{varest}
\int_{\mathbb{R}^d}K^{2}\left(\frac{t-x}{h}\gamma(x)\right) \gamma^{2d}(x)f(x)dx
&=&h^d\int_{\mathbb{R}^d}K^{2}\left(u\gamma(t-hu)\right)\gamma^{2d}(t-hu)f(t-hu)du\nonumber\\
&\le &
2^d(2T^{1/2}/\nu)^{d}||K||_\infty^2||f||_\infty||\alpha||_\infty^{2d} h^d.\label{var2}
\end{eqnarray}
So, assuming $f\in{\cal P}_C$, we can take $\sigma_h^2:=C(K,\alpha)Ch^d$ with $C(K,\alpha)=2^d(2T^{1/2}/\nu)^{d}||K||_\infty^2||\alpha||_\infty^{2d}$.
Take now $h=h_n$. The envelope
$U_{h_n}$ of ${\cal F}_{h_n}$ can be taken to be the constant $U$ above, hence it is eventually much larger than $\sigma_{h_n}$ and,  by (\ref{band}), we also have
$\sqrt{n}\sigma_{h_n}\sqrt{\log (U/\sigma_{h_n})}<<n\sigma_{h_n}^2$
(here and elsewhere, the sign $<<$ should be read as `of smaller order than' when the indexing variable, in this case $k$, tends to infinity). If $C_1$ and $C_2$ are the constants in Talagrand's inequality (\ref{tal})  common to all the classes ${\cal F}_h$, it is then clear that there is $n\ge n_0$, $n_0$ large enough, so that there exists $\lambda>0$ such that simultaneously,
\begin{equation}\label{tcond}
C_1\sqrt{n}\sigma_{h_n}\sqrt{\log \frac{AU}{\sigma_{h_n}}}<\lambda\sqrt{nh_n^d\log n}<<n\sigma_{h_n}^2,
\end{equation}
for all $n\ge n_0$, and $\lambda>\sqrt{C_2C(K,\alpha)C}$ (note that $\log(AU/\sigma_{h_n}$ is of the order of a constant times $\log n$). Then, applying Talagrand's inequality (\ref{tal})  to the empirical process in (\ref{eq1})  gives
\begin{equation}\label{tineq}
\sum_n\sup_{f\in{\cal P}_C}\Pr\left\{\sqrt{\frac{n h_{n}^d}{\log n}}
||\bar{f}_n-E\bar{f}_n||_{\infty}>\lambda\right\}\le C_2\sum_n \exp\left(-\frac{\lambda^2 \log n}{C_2C(K,\alpha)C }\right)<\infty.
\end{equation}
This proves the proposition.
\end{proof}

\medskip
Here is an alternative proof: 

\medskip
\begin{proof} The facts  that the union of the classes ${\cal F}_h$, $0<h<1/(2\|\beta\|_\infty)^{1/2}$, is VC bounded and measurable, and that inequality (\ref{varest}) holds, verify that  the class of functions
$\{K((t-\cdot)h^{-1}\gamma(\cdot))\gamma^d(\cdot):t\in\mathbb R^d,0<h<1/(2\|\beta\|_\infty)^{1/2}\}$ satisfies the hypotheses of the general theorem in Mason and Swanepoel (2010), and their result then implies a stronger version of Proposition \ref{varid} with $h_n$ replaced by $h$ and with uniformity in $h$ within a range that includes $h_n$. Their theorem is not stated with uniformity in $f\in{\cal P}_C$, but the inequalities used in their proof imply it. \end{proof}

\medskip
Corollary \ref{4} in Subsection 2.1 shows that  the bias of the ideal estimator $f_{McK}(t;h_{2,n})$ from (\ref{idealest}) is of the order of $h_{2,n}^4$ uniformly in $t\in {\cal D}_r$ and in $f\in{\cal P}_{C,4}$,  and Proposition  \ref{varid} (with $\beta\equiv 0$ in the definition of $\gamma_h$) gives that the uniform deviation from its mean,  $\sup_{t\in \mathbb{R}^d}|f_{McK}(t;h_{2,n})-Ef_{McK}(t;h_{2,n})|$, has order $O_{\rm a.s.}\left(\sqrt{\frac{\log n}{n h_{2,n}^d}}\right)$ uniformly in $t\in \mathbb R^d$ and in $f\in{\cal P}_C$ (any $0<C<\infty$). Hence, they are of the same order for $h_{2,n}=((\log n)/n)^{1/(8+d)}$, and we have
\begin{equation}\label {idealrateK}
\sup_{t\in {\cal D}_r}|f_{McK}(t;h_{2,n})-f(t)|=O_{\rm a.s.}\left(((\log n)/n)^{4/(8+d)}\right)\ \ {\rm uniformly\ in}\ f\in{\cal P}_{C,4}.
\end{equation}
Likewise, Corollary \ref{6} and Proposition  \ref{varid} give that, for $h_{2,n}=((\log n)/n)^{1/13}$
\begin{equation}\label{idealrateJ}
\sup_{t\in {\cal D}_r}|f_{JKH}(t;h_{2,n})-f(t)|=O_{\rm a.s.}\left(((\log n)/n)^{6/13}\right)\ \ {\rm uniformly\ in}\ f\in{\cal P}_{C,6}.
\end{equation}

It is this balance between the bias and the random centered (variance) components of the difference $f_{McK}(t;h_{2,n})-f(t)$ (or $f_{JKH}(t;h_{2,n})-f(t)$) that prevents us from taking advantage  of the uniformity in bandwidth in the Mason-Swanepoel (2010) theorem. Note also that applying this result as in the second proof of Proposition \ref{varid} and applying Talagrand's inequality (\ref{tal}) as in its first proof,  both require checking exactly the same facts (namely, measurability, boundedness and VC character of a class of functions, plus a bound for the variances of the functions in the class). For these two reasons we continue using Talagrand's inequality in the few instances below where either works.

\section{Estimation of densities in $C^4(\mathbb R^d)$: Proof of Theorem \ref{main0}}
In this section we develop the proof of Theorem \ref{main0}. The pattern of proof is similar to that of the main results in Hall and Marron (1988), corrected in Hall, Hu and Marron (1995), and particularly  in Gin\'e and Sang (2010), but the details are quite different.

We make the following assumptions on the kernel $K$, the clipping function $p$, the densities $f$ and the bandsequences:

\medskip
\begin{assumptions}\label{ass3}

The kernel $K$ is assumed to satisfy all the conditions in Proposition \ref{bias0} and Assumptions \ref{ass1},   and to have, besides, uniformly bounded second order partial derivatives.  We also assume that the densities $f$ are bounded and have at least four bounded and uniformly continuous derivatives, that is, $f\in{\cal P}_{C,4}$ for some $C<\infty$. The nondecreasing clipping function $p:\mathbb {R}\rightarrow \mathbb {R}$ is assumed to have two bounded derivatives, $p(s)\ge 1$ for all $s$ and $p(s)=s$ for all $s\ge t_0\ge 1$.  Here $c$ and $t_0$ are fixed constants.We set $h_{1,n}=((\log n)/n)^{1/(4+d)}$ and $h_{2,n}=((\log n)/n)^{1/(8+d)}$, $n\in\mathbb N$.
 \end{assumptions}

Following Hall and Marron (1988), we compare the ideal estimator (\ref{idealest}) with the real one (\ref{realest}), with some changes similar to the ones introduced in  Gin\'e and Sang (2010), particularly, the use of inequalities from empirical and $U$-processes. Proving the theorem in $\mathbb R^d$ for any $d>0$ requires more precision than in dimension 1, in particular we cannot undersmooth the preliminary estimator and we must proceed differently with several estimations.

By (\ref{idealrateK}, in order to obtain a uniform convergence rate of $((\log n)/n)^{4/(8+d)}$  for teh difference between the true estimator (\ref{realest}) and the density $f(t)$ we only need to show that the uniform convergence rate of the difference between the true estimator (\ref{realest}) and the ideal estimator (\ref{idealest}) is at most of this order.

Recall  $\alpha(t):=cp^{1/2}(c^{-2}t)$. Define
$\delta(t)=\delta(t,n)$ by the equation
\begin{equation}\label{delta}
\delta(t)=\frac{\alpha(\hat f(t;h_{1,n}))-\alpha (f(t))}{\alpha(f(t))}=\frac{p(c^{-2}\hat f(t;h_{1,n}))-p(c^{-2}f(t))}{p^{1/2}(c^{-2}f(t))[p^{1/2}(c^{-2}\hat f(t;h_{1,n}))+p^{1/2}(c^{-2}f(t))]},
\end{equation}
so that
\begin{equation}\label{alphahat}
\alpha(\hat f(t;h_{1,n}))=\alpha(f(t))(1+\delta(t)).
\end{equation}

Since $p$ is a Lipschitz function and $p\ge 1$,
\begin{equation}\label{deltabd}
|\delta(t)|\le B c^{-2}|\hat f(t;h_{1,n})-f(t)|
\end{equation}
for a constant $B$ that depends only on $p$. Set
$$D(t;h_{1,n})=\hat f(t;h_{1,n})-E\hat f(t;h_{1,n})\ \ {\rm and}\ \ b(t;h_{1,n})=E\hat f(t;h_{1,n})-f(t)$$
and note that
\begin{equation}\label{classic1}
\|D(\cdot;h_{1,n})\|_\infty=O_{a.s.}\left(\sqrt{\frac{\log h_{1,n}^{-1}}{nh_{1,n}^d}}\right)\ \ {\rm uniformly\ in}\ f\in{\cal P}_C
\end{equation}
for all $0<C<\infty$ by a result of  Gin\'e and Guillou (2002), and that
\begin{equation}\label{classic2}
 \|b(\cdot;h_{1,n})\|_\infty=O_{\rm a.s.}(h_{1,n}^2)\ \ {\rm uniformly\ in}\ f\in{\cal P}_{C,2}
\end{equation}
by the classical bias computation for symmetric kernels. Then
we have, by (\ref{deltabd}), (\ref{classic1}) and (\ref{classic2}),
\begin{equation}\label{zero}
\sup_{t\in \mathbb{R}^d}|\delta(t)|=O_{\rm a.s.}(h_{1,n}^2)=o_{\rm a.s.}(1)\ \ {\rm uniformly\ in}\ f\in{\cal P}_{C,2}.
\end{equation}
We also have, for further use,
\begin{eqnarray}\label{deltatalor}
\delta(t)&=&\frac{\alpha(\hat f(t;h_{1,n}))-\alpha(f(t))}{\alpha(f(t))}\nonumber\\
&=&\frac{\alpha'(f(t))[\hat f(t;h_{1,n})-f(t)]
}{\alpha(f(t))}
+\frac{\alpha''(\eta)[\hat f(t;h_{1,n})-f(t)]^2}{2\alpha(f(t))}\end{eqnarray}
where $\eta=\eta(t)\ge 0$ is between $\hat f(t;h_{1,n})$ and $f(t)$ (so, not only it depends on $t$ but also on $f$ and on the whole sample. Note that, since $p\ge 1$ and $p'$ and $p''$ are uniformly bounded on $[0,\infty)$, we have $|\alpha''(\eta( t,h_{1,n}))|\le c^{-3}A$ for some constant $A$ that does not depend on $n$ or $t$ but only on $p$.  It is convenient as well to record the following expansion of $\alpha^d(\hat f)$ implied by (\ref{alphahat}) and (\ref{zero}):
\begin{equation}\label{delta1}
\alpha^d(\hat f(t;h_{1,n}))=\alpha^d(f(t))(1+d\delta(t))+\delta_1(t)
\end{equation}
with\begin{equation}\label{delta1'}
\|\delta_1\|_\infty=O_{\rm a.s.}(\|\delta\|_\infty^2)\ \ {\rm uniformly\ in}\ f\in{\cal P}_{C.2},
\end{equation}
hence, by (\ref{deltabd}) and (\ref{zero}), 
\begin{equation}\label{delta1bd}
\|\delta_1\|_\infty=O_{\rm a.s.}(\|\hat f_n(\cdot;h_{1,n})-f(\cdot)\|_\infty^2)\ \ {\rm uniformly\ in}\  f\in{\cal P}_{C,2}.
\end{equation}
For the kernel $K$ we have the expansion
\begin{eqnarray}\label{Kexp}
K\left(\frac{t-X_i}{h_{2,n}}\alpha(\hat f(X_i;h_{1,n}))\right)\!\!\!&=&\!\!\!K\left(\frac{t-X_i}{h_{2,n}}\alpha(f(X_i))\right)
\\
&&\!\!+\sum_{j=1}^dK'_j\left(\frac{t-X_i}{h_{2,n}}\alpha(f(X_i))\right)\frac{(t-X_i)_j}{h_{2,n}}\alpha(f(X_i))\delta(X_i)+\delta_2(t;X_i)\notag
\end{eqnarray}
with 
$$\delta_2(t,x):=\sum_{j,\ell=1}^dK''_{j,\ell}(\xi)\frac{(t-x)_j(t-x)_\ell}{2h_{2,n}^2}\alpha^2(f(x))\delta^2(x),$$
where $\xi$ is a (random) point in the line connecting the points $\frac{t-X_i}{h_{2,n}}\alpha(f(X_i))$ and
$\frac{t-X_i}{h_{2,n}}\alpha(f(X_i))+
\frac{t-X_i}{h_{2,n}}\alpha(f(X_i))\delta (X_i),$ as before. $K$ having compact support, $\alpha$ being bounded from below by $c$ (and above on bounded sets) and $\delta$ satisfying (\ref{zero}) and ({\ref{deltabd}), we get
that, for each $n$, on the set where $\|\hat f_n(\cdot;h_{1,n})-f(\cdot)\|_\infty^2\le c^2/(2B)$ (so, $\|\delta\|_\infty\le 1/2$),
\begin{equation}\label{delta2bd}
|\delta_2(t,x)|\le \frac{d^2\|K''\|_\infty}{2}(2T^{1/2}/c)^2\delta^2(x)I(\|t-x\|\le 2T^{1/2}c^{-1}h_{2,n}).
\end{equation}
in particular,
\begin{equation}\label{delta2bd'}
\sup_{t,x\in\mathbb R^d}|\delta_2(t,x)|=O_{\rm a.s.}\left(\|\hat f_n(\cdot;h_{1,n})-f(\cdot)\|_\infty^2\right)\ \ {\rm uniformly\ in}\ \ f\in{\cal P}_{C,2}.
\end{equation}
Set 
\begin{equation}\label{L1Lfunctions}
L_1(t)=\sum_{i=1}^dt_iK'_i(t)\ \ {\rm and}\ \  L(t)=dK(t)+L_1(t),\ \ t\in{\mathbb R}^d,
\end{equation}
and notice that by symmetry, integration by parts gives that $L$ is a second order kernel, as in dimension 1 ($K'_j$ denotes the partial derivative of $K$ in he direction of the $i$-th coordinate, and $t_i$ dentoes the $i$-the coordinate of $t\in\mathbb R^d$).
The decompositions (\ref{alphahat}), (\ref{delta1}) and (\ref{Kexp})  then give:
\begin{eqnarray}
{nh_{2,n}^d}\hat f(t;h_{1,n},h_{2,n})\!\!\!\!&=&\!\!\!\!{nh_{2,n}^d}\bar f(t;h_{2,n})\notag\\
&&\!\!+\sum_{i=1}^{n}L\left(\frac{t-X_i}{h_{2,n}}\alpha(f(X_i))\right)
\alpha^d(f(X_i))\delta (X_i)\label{one}\\
&&\!\!+\sum_{i=1}^{n}\Bigg[K\left(\frac{t-X_i}{h_{2,n}}\alpha(f(X_i))\right)\delta_1(X_i)+\alpha^d(f(X_i))\delta_2(t,X_i)\notag\\
&&~~~~~~~~~~~~~~~~~~~~~+dL_1\left(\frac{t-X_i}{h_{2,n}}\alpha(f(X_i))\right)\alpha^d(f(X_i))\delta^2(X_i)\Bigg]\label{two}\\
&&\!\!+\sum_{i=1}^{n}\left[L_1\left(\frac{t-X_i}{h_{2,n}}\alpha(f(X_i))\right)\delta(X_i)\delta_1(X_i)+d\alpha^d(f(X_i))\delta(X_i)\delta_2(t,X_i)\right]\notag\\
&&\!\!\label{three}\\
&&\!\!+\sum_{i=1}^n \delta_2(t,X_i)\delta_1(X_i)\label{four}
\end{eqnarray}
The sums (\ref{two})-(\ref{four}) are of lower and decreasing order, and will be dealt with first. Here is an elementary observation that will be useful: if for all $n\ge r$  we have that $U_n\le V_n$ whenever $W_n\le a$, then for all $\alpha$ and for
 $k\ge r$,
\begin{equation}\label{elem}
\Pr\left\{\sup_{n\ge k}U_n\ge\alpha\right\}\le \Pr\left\{\sup_{n\ge k}V_n\ge\alpha\right\}+\Pr\left\{\sup_{n\ge k}W_n\ge a\right\}.
\end{equation}
We will use this in conjunction with the fact that there is $a<\infty$ such that 
\begin{equation}\label{classic3}
\sup_{f\in{\cal P}_{C,2}}{\Pr}_f\left\{\sup_{n\ge k}\sqrt{\frac{nh_{1,n}^d}{\log n}}\|\hat f(t;h_{1,n})-f\|_\infty\ge a\right\}\to 0\ \ {\rm as}\ k\to\infty,
\end{equation}
which follows from (\ref{classic1}) and (\ref{classic2}) since for our value of $h_{1,n}$, we have $h_{1,n}^4=(\log n)/(nh_{1,n}^d)$. Let us consider the first term from (\ref{two}): by the previous observation (\ref{elem}),
\begin{eqnarray}\label{res}
&&\!\!\!\!\Pr\left(\sup_{n\ge k}\frac{1}{nh_{2,n}^dh_{1,n}^4}\sum_{i=1}^{n}\left|K\left(\frac{t-X_i}{h_{2,n}}\alpha(f(X_i))\right)\delta_1(X_i)\right|>\alpha\right)\\&&\le\Pr\left(\sup_{n\ge k}\frac{1}{nh_{2,n}^d}\left\|\sum_{i=1}^{n}\left|K\left(\frac{t-X_i}{h_{2,n}}\alpha(f(X_i))\right)\right|\right\|_\infty>\alpha/b\right)
+\Pr\left(\max_{n\ge k}\|h_{1,n}^{-4}\delta_1\|_\infty>b\right),\notag
\end{eqnarray}
 and the last term, for suitable $b<\infty$, converges to zero uniformly in $f\in{\cal P}_{C,2}$ as $k\to\infty$ by (\ref{classic3}) and (\ref{delta1bd}). Now, by change of variables, for any $r>0$,
 $$E\left|K^r\left(\frac{t-X_i}{h_{2,n}}\alpha(f(X_i))\right)\right|\le C(K,r)\|f\|_\infty h_{2,n}^d$$
 for some finite constant $C(K,r)$, so that, for a suitable constant $m$, the first term at the right hand side of  (\ref{res}) is bounded by
 $$\sum_{n=k}^\infty\Pr\left(\frac{1}{nh_{2,n}^d}\left\|\sum_{i=1}^{n}\left(\left|K\left(\frac{t-X_i}{h_{2,n}}\alpha(f(X_i))\right)\right|-E\left|K\left(\frac{t-X_1}{h_{2,n}}\alpha(f(X_1))\right)\right|\right)\right\|_\infty>\alpha/b-m\right).$$
Here we can use Talagrand's inequality (\ref{tal})  (if a class of functions is $VC$ type, so is the class of its absolute values, by direct computation of $L_2$ distances), which, by the second moment estimate above ($r=2$) and boundedness of $K$, and for suitable $\alpha$ (in particular making $\alpha/b-m>0$), shows that this series is dominated, uniformly if $f\in{\cal P}_{C}$, by 
$\sum_{n\ge k}e^{-\eta nh_{2,n}^d}$ for some $\eta>0$, which tends to zero. (Note that we are using Talagrand's inequality for $t$ in the upper limit of its domain (\ref{talc}), whereas typically one uses it for $t$ in its lower limit.) Thus, we have proved that, uniformly in $f\in{\cal P}_{C,2}$, 
$$\frac{1}{nh_{2,n}^d}\sup_{t\in\mathbb R^d}\sum_{i=1}^{n}\left|K\left(\frac{t-X_i}{h_{2,n}}\alpha(f(X_i))\right)\delta_1(X_i)\right|=O_{\rm a.s.}(h_{1,n}^4)=o_{\rm a.s.}\left((\log n)/n)^{4/(8+d)}\right).$$
Basically, what Talagrand and the simple observation (\ref{elem}) do is to show that the order of the first term in (\ref{two}) is at most the order of $\|\delta_1\|_\infty$ multiplied by the order of the sup of the expectations of the (absolute values of) the summands without $\delta_1(X_i)$. We can likewise follow this pattern of proof and get similar results for all the terms in (\ref{two})-(\ref{four}). One has to use that the classes of functions $\{L_1\left(\frac{t-\cdot}{h}\alpha(f(\cdot)):t\in\mathbb R^d, h>0\right\}$ and $\left\{I(\|t-\cdot\|\le 2T^{1/s}c^{-1}h):t\in\mathbb R^d, h>0\right\}$ are VC type by Lemma \ref{vc1} (the class of indicator functions is needed in order to handle the three terms in (\ref{two})-(\ref{four}) that contain $\delta_2$). We then get
\begin{equation}\label{easyterms}
\frac{1}{nh_{2,n}^d}\sup_{t\in\mathbb R^d}|(\ref{two})|=O_{\rm a.s.}(h_{1,n}^4), \ \frac{1}{nh_{2,n}^d}\sup_{t\in\mathbb R^d}|(\ref{three})|=O_{\rm a.s.}(h_{1,n}^6),\ \ \frac{1}{nh_{2,n}^d}\sup_{t\in\mathbb R^d}|(\ref{four})|=O_{\rm a.s.}(h_{1,n}^8)
\end{equation}
uniformly in $f\in{\cal P}_{C,2}$.

The estimation of (\ref{one}) is much more difficult. We decompose into several pieces using first the expansion  (\ref{deltatalor}) of $\delta$, and then the decomposition of $\hat f-f$ into variance $D$ and bias $b$, as follows:
\begin{eqnarray}
&&\!\!\!\!\!\!\!\!\!\!\frac{1}{n h_{2,n}^d}\sum_{i=1}^{n}L\left(\frac{t-X_i}{h_{2,n}}\alpha(f(X_i))\right)
\alpha^d(f(X_i))\delta (X_i)\nonumber\\
&&\!\!=\frac{1}{dn h_{2,n}^d}\sum_{i=1}^{n}L\left(\frac{t-X_i}{h_{2,n}}\alpha(f(X_i))\right)
(\alpha^d)'( f(X_i)) D(X_i;h_{1,n})\label{eps1}\\
&&+\frac{1}{dn h_{2,n}^d}\sum_{i=1}^{n}L\left(\frac{t-X_i}{h_{2,n}}\alpha(f(X_i))\right)
(\alpha^d)'(f(X_i))  b(X_i;h_{1,n})\label{eps2}\\
&&+\frac{1}{2n h_{2,n}^d}\sum_{i=1}^{n}L\left(\frac{t-X_i}{h_{2,n}}\alpha(f(X_i))\right)\alpha^{d-1}(fX_i))(\alpha^d)''(\eta(X_i))[\hat f(X_i;h_{1,n})-f(X_i)]^2\!.\label{e5}
\end{eqnarray}
Notice that the term (\ref{e5}) is very similar to the terms in (\ref{two}), and it has clearly the same order (recall (\ref{deltabd})), that is
\begin{equation}\label{easyterm2}
\sup_{t\in\mathbb R^d}|(\ref{e5})|=O_{\rm a.s.}(h_{1,n}^4)\ \ {\rm uniformly\ in}\ \ {\cal P}_{C,2}.
\end{equation}
We devote two subsections to the estimation of the remaining two terms. we anticipate that the main term is (\ref{eps1}).

\medskip
\noindent{\bf 3.1. Estimation of the bias term (\ref{eps2}).} Consider the classes of functions
\begin{equation}\label{qu}
{\cal Q}_n:=\left\{Q(x)=L\left(\frac{t-x}{h_{2,n}}\alpha(f(x))\right)
(\alpha^d)'( f(x))  b(x;h_{1,n}):t\in \mathbb{R}^d\right\}.
\end{equation}
Recall that $L(t)=\sum_{i=1}^dt_iK'_i+dK(t)$ and that    $K(t)=\Phi(\|t\|^2)$,  $\Phi$ twice boundedly differentiable and with bounded support. Hence, $L(t)=2\|t\|^2\Phi'(\|t\|^2)+d\Phi(\|t\|^2)$. Since the function 
$u\Phi'(u)+d\Phi(u)$ is of bounded variation and bounded, the kernel $L$ satisfies the hypotheses of the kernel $K$ in Lemma \ref{vc1} (with $s=2$). So, these classes conform,
for each $n$,  to Lemma \ref{vc1} with $\cal G$ the class consisting of the single function  $(\alpha^d)'(f(x))b(x,h_{1,n})$, which, by (\ref{classic2}), is uniformly bounded by $M(c,p,C)h_{1,n}^2$ if $f\in{\cal P}_{C,2}$, for some constant $M$ depending only on $c$, $p$ and $C$. We conclude by that lemma that they are $VC$ each with with envelope $M(c,p, C, K)h_{1,n}^2$ for some other constant depending on the stated objects, and all with the same characteristic constants $A$ and $v$. 
Since the continuity hypotheses make these classes measurable, this will allow us to apply Talagrand's inequality.
 If we set
 $$Q_i(t)=L\left(\frac{t-X_i}{h_{2,n}}\alpha(f(X_i))\right)
(\alpha^d)'( f(X_i))  b(X_i;h_{1,n})$$
it then follows, by the bound (\ref{classic2}) on $b$, by boundedness and bounded support of $L$, by boundedness of $p^\prime$  and $p\ge 1$,  that, for all $t$ and all $f\in{\cal P}_{C,2}$,
\begin{eqnarray*}
\sup_{t\in\mathbb R^d}EQ_i^2(t)&\le& \|b(\cdot;h_{1,n})\|_\infty^2\|(\alpha^d)'\|^2_\infty\|f\|_\infty h_{2,n}^d
\sup_{t\in\mathbb R^d}\int_{\mathbb R^d} L^2(u\alpha(f(t-uh_{2,n})))du\\
&\le & M(C,c,p,K)h_{1,n}^4h_{2,n}^d,
\end{eqnarray*}
and similarly,
$$\sup_{t\in \mathbb{R}^d}|Q_i(t)|\lessim \bar M(c,p, C,K) h_{1,n}^2.$$
We then have
\begin{equation*}
\sup_{t\in \mathbb{R}^d}\left|\frac{1}{nh_{2,n}^d}\sum_{i=1}^nQ_i(t)\right|
\le \sup_{t\in \mathbb{R}^d}\left|\frac{1}{nh_{2,n}^d}\sum_{i=1}^{n} [Q_i(t)-EQ_i(t)]\right|+ \sup_{t\in \mathbb{R}^d}\frac{1}{h_{2,n}^d}|EQ_1(t)|
\end{equation*}
and Talagrand's inequality (\ref{tal}), with $\sigma^2=M(C,c,p,K)h_{1,n}^4h_{2,n}^d$ and $F=2M(c,p, C,K) h_{1,n}^2$ gives that, for some $D$, $L>1$,
$$\sum_n\sup_{f\in{\cal P}_C}{\Pr}_f\left\{\sup_{t\in \mathbb{R}^d}\left|\sum_{i=1}^{n} [Q_i(t)-EQ_i(t)] \right|\ge D\sqrt{nh_{1,n}^4h_{2,n}^d\log n}\right\}\le
C_2\sum_n\exp\left(-L\log n\right)<\infty.$$
Since $\sqrt{nh_{1,n}^4h_{2,n}^d\log n}/(nh_{2,n}^d)<<((\log n)/n)^{4/(8+d)}$, the term  (\ref{eps2}) will be at most of the order of $((\log n)/n)^{4/(8+d)}$ only if the expectation term $\sup_{t\in \mathbb{R}^d}\frac{1}{h_{2,n}^d}|EQ_1(t)|$ is of this order or smaller (uniformly in $t\in\mathbb R^d$ and $f\in{\cal P}_{C,2}$). The obvious bound for $E|Q_1(t)|$, that one obtains just like  the bound above for $EQ_1^2(t)$, is of the order of $h_{1,n}^2h_{2,n}^d$, which then gives an order of $h_{1,n}^2$ for the term (\ref{eps2}). This is not good enough, although it would be if we undersmoothed the preliminary estimator a little by taking $h_{1,n}=((\log n)/n)^{2/(8+d)}$ instead of $h_{1,n}=((\log n)/n)^{1/(4+d)}$: this works, and in fact we made this choice in Gin\'e and Sang (2010) on a related problem and $d=1$, however, some extra work along the lines suggested by Hall and Marron (1988) will allow us to prove the right rate for (\ref{eps2}) with the optimal $h_{1,n}$, as follows. 
In the setting of the proof of Proposition \ref{bias0}, but in dimension $d$, the inverse function theorem yields 
the existence and differentiability of $V_t(u)$, the inverse function of $U_t(v)=v\alpha(f(t-v))$ in a neighborhood of zero independent of $t$ (this can be readily seen, or one can see it in McKay (1993b)). This, together with the facts that $L$ has bounded support, $\alpha$ is bounded away from zero and $h_{2,n}\to0$, justifies the change of variables $hz=(t-s)\alpha(f(t-(t-s)))$ in the expression of $EQ_1(t)$, to get (omitting the subindex $n$ in the bandwidth),
\begin{eqnarray*}
\frac{1}{h_2^d}EQ_1(t)&=&\frac{1}{h_2^d}\int_{\mathbb R^d} L\left(\frac{t-s}{h_2}\alpha(f(s))\right)(\alpha^d)'(f(s))f(s)\int_{\mathbb R^d}\frac{1}{h_1^d}K\left(\frac{s-u}{h_1}\right)(f(u)-f(s))duds\\
&=&-\int_{\mathbb R^d}\bigg((\alpha^d)'(f(t-V_t(h_2z)))f(t-V_t(h_2z))\left(\frac{\partial V_t}{\partial v}\right)_{v=h_2z}\\
&&~~~~~~~~~\times\int_{\mathbb R^d} K(y)\left(f(t-V_t(h_2z)-yh_1)-f(t-V_t(h_2z))\right)dy\Bigg)L(z)dz\\
&:=&\int_{\mathbb R^d} F(h_2z)G(h_2z)L(z)dz.
\end{eqnarray*}
Then, using that $\int L(z)dz=\int z_iL(z)dz=0$ and expanding, we obtain
$$\frac{1}{h_2^d}EQ_1(t)=-\frac{h_2^2}{2}\int_{\mathbb R^d}\left(\sum_{i,j=1}^d(F''_{i,j}G+F'_iG'_j+F'_jG'_i+FG''_{i,j})(\theta(h_2z))z_iz_j\right)L(z)dz.$$
Now, $F$ and its partial derivatives are bounded, $L$ is bounded and has bounded support,  and, if $f\in {\cal P}_{C,4}$, then, expanding $g(t-V_t(h_2z)-yh_1)-g(t-V_t(h_2z))$, for $g=f, f'_i, f''_{i,j}$, and using the symmetry of $K$, we get that 
$\|G\|_\infty$, $\|G'_i\|_\infty$, $\|G''_{i,j}\|_\infty$ are all $O(h_1^2)$ uniformly in $f\in{\cal P}_{C,4}$. We conclude
\begin{equation}\label{biasbias}
\sup_{t\in\mathbb R^d}\frac{1}{h_{2,n}^d}|EQ_1(t)|=O(h_{1,n}^2h_{2,n}^{2}) \ \ {\rm uniformly\ in} \ f\in{\cal P}_{C,4},
\end{equation}
and this in turn gives, together with the above application of Talagrand's inequality,
\begin{equation}\label{bb}
\sup_{t\in \mathbb{R}^d}\left|\frac{1}{nh_{2,n}^2}\sum_{i=1}^nL\left(\frac{t-X_i}{h_{2,n}}\alpha(f(X_i))\right)
(\alpha^d)'( f(X_i))  b(X_i;h_{1,n})\right|=o_{\rm a.s.}(n^{-4/(8+d)})
\end{equation}
uniformly in $f\in{\cal P}_{C,4}$.
\medskip

\noindent{\bf 3.2. Estimation of the variance term (\ref{eps1}).} This term requires $U$-processes. Given a function $H$ of two variables, and two i.i.d. random variables $X$ and $Y$ such that $H(X,Y)$ is integrable, we recall the $U$-statistic notation
$$U_n(H)=\frac{1}{n(n-1)}\sum_{1\le i\ne j\le n}H(X_i,X_j),$$
where the variables $X_i$ are i.i.d. copies of $X$, as well as the second order Hoeffding projection of $H(X,Y)$,
$$\pi_2(H)(X,Y)=H(X,Y)-E_XH(X,Y)-E_YH(X,Y)+EH,$$
If we set
\begin{equation}\label{H}
H_t(X,Y):=L\left(\frac{t-X}{h_{2,n}^d}\alpha(f(X))\right)
(\alpha^d)'(f(X))K\left(\frac{X-Y}{h_{1,n}}\right),
\end{equation}
then (\ref{eps1}) decomposes into a diagonal term and a $U$-statistic term, as follows:
\begin{eqnarray}\label{hoef1}
&&\frac{n^2h_{1,n}^dh_{2,n}^d}{n(n-1)}\frac{1}{nh_{2,n}^d}\sum_{i=1}^{n}L\left(\frac{t-X_i}{h_{2,n}}\alpha(f(X_i))\right)
(\alpha^d)'( f(X_i)) D(X_i;h_{1,n})\notag\\
&&=\frac{1}{n(n-1)}\sum_{i=1}^n (H_t(X_i,X_i)-E_YH_t(X_i,Y))+U_n(H_t-E_YH_t(\cdot,Y))\notag\\
&&=\frac{1}{n(n-1)}\sum_{i=1}^n (H_t(X_i,X_i)-E_YH_t(X_i,Y))+U_n\left(\pi_2(H_t(\cdot,\cdot))\right)\notag\\
&&~~~~~~~~~~~~~~~~~~~~~~~~~~~~~+\frac{1}{n}\sum_{i=1}^n(E_XH_t(X, X_i)-EH_t).
\end{eqnarray}
These are two empirical process terms and a canonical $U$-statistic term. The last term will turn out to be the only significant one.

For the first empirical process in (\ref{hoef1}), set
$$\bar Q_i(t)=H_t(X_i,X_i)-E_YH_t(X_i,Y)$$
and observe that, very much as in the simple bounds for moments of  $Q_i$ in the previous subsection,
$$\sup_{f\in{\cal P}_C}
\sup_{t\in \mathbb{R}^d}E|\bar Q_1(t)|\le L_1 h_{2,n}^d,\ \ \sup_{f\in{\cal P}_C}\sup_{t\in \mathbb{R}^d}E\bar Q_1^2(t)\le L_2 h_{2,n}^d, \ \ \sup_{f\in{\cal P}_C}\sup_{t\in \mathbb{R}^d}|\bar Q_1(t)|\le L_3$$
for some finite constants $L_i=L_i(C,c, p, K)$.
So,
\begin{equation}\label{interm}
\sup_{t\in \mathbb{R}^d}\frac{1}{n^2h_{1,n}^dh_{2,n}^d}\left|\sum_{i=1}^n\bar Q_i(t)\right|\le\frac{1}{n^2h_{1,n}^dh_{2,n}^d}\sup_{t\in \mathbb{R}^d}\left|\sum_{i=1}^n(\bar Q_i(t)-E\bar Q_1(t))\right|+\frac{L_1}{nh_{1,n}^d}.
\end{equation}
The supremum part corresponds to the empirical process over the class of functions of $x$
\begin{equation*}
\bar{\cal Q}_n=\left\{L\left(\frac{t-x}{h_{2,n}}\alpha(f(x))\right)
(\alpha^d)'(f(x))\left(K(0)-EK\left(\frac{x-X}{h_{1,n}}\right)\right)
: t\in \mathbb{R}^d\right\},
\end {equation*}
which, by Lemma \ref{vc1} is of VC type with respect to a constant envelope and admits characteristic constants $A$ and $v$ independent of $n$ and $f$, just as in the previous subsection for the classes ${\cal Q}_n$ defined by (\ref{qu}). Then, as in this previous instance, Talagrand's inequality (\ref{tal}) gives that there exist $D_1,D_2>1$ such that
$$\sum_n\sup_f{\Pr}_f\left\{\sup_{t\in \mathbb{R}^d}\left|\sum_{i=1}^n(\bar Q_i(t)-E\bar Q_1(t))\right|>D_1\sqrt{nh_{2,n}^d\log n}\right\}\le
C_2\sum_n\exp\left(-D_2\log n\right)<\infty,$$
which, since $\sqrt{nh_{2,n}^d\log n}/(n^2h_{1,n}^dh_{2,n}^d)<<((\log n)/n)^{4/(8+d)}$ and since also \hfil\break$nh_{1,n}^d>> (n/\log n)^{4/(8+d)}$, together with (\ref{interm})
yields
\begin{equation}\label{diag}
\sup_{t\in \mathbb{R}^d}\frac{1}{n^2h_{1,n}^dh_{2,n}^d}\left|\sum_{i=1}^n (H_t(X_i,X_i)-E_YH_t(X_i,Y))\right|=o_{a.s.}((\log n)/n)^{4/(8+d)}) \ {\rm uniformly\ in} \ f\in{\cal P}_{C}.
\end{equation}

The canonical $U$-statistic term in (\ref{hoef1}) is best handled by means of an exponential inequality of Major (2006). We will state his inequality for bounded VC type classes of functions of two variables only. Let $\cal F$ be such a class of functions and let $\|F\|_\infty^2\ge \sigma^2\ge \|{\rm Var}(f(X_1,X_2))\|_{\cal F}$.  Then, if $\cal F$ is a uniformly bounded, countable class of VC type, there exist  $0<C_i<\infty$, $1\le i\le 3$, depending on $v$ and $A$ such that, for all $t$ satisfying
$$C_1n\sigma\log\frac{2\|F\|_\infty}{\sigma}\le t\le \frac{n^2\sigma^3}{\|F\|_\infty^2}$$
we have
\begin{equation}\label{major}
\Pr\left\{\left\|\sum_{1\le i\ne j\le n}\pi_2^Pf(X_i,X_j)\right\|_{\cal F}>t\right\}\le C_2\exp\left(-C_3\frac{t}{n\sigma}\right).
\end{equation}
Major states the theorem for $\{\pi_2^Pf\}$ of VC type, but it is easy to see that if $\cal F$ is VC type for $F$ then $\{\pi_2^Pf:f\in {\cal F}\}$ is VC type for the envelope $4F$. Our classes $\cal F$ will be the classes $\{H_t:t\in\mathbb R^d\}$. Note that they depend on $n$ via $h_{i,n}$, $i=1,2$, but we do not display this dependence because they are VC type for a fixed constant envelope, admitting characteristic constants $A$ and $v$ independent of $n$: this follows from Lemma \ref{vc1} with $L$ instead of $K$, and with $\cal G$ consisting of the single bounded function $(\alpha^d)'(f(X))K\left(\frac{X-Y}{h_{1,n}}\right)$ (see Section 3.1, proof that the classes defined in (\ref{qu}) are $VC$). Since, as is easy to check, for $f\in{\cal P}_C$,
$$EH_t^2(X,Y)\le M(c,p,C,K) h_{1,n}^dh_{2,n}^d,$$
we can take   $\sigma^2=  M(c,p,C,K) h_{1,n}^dh_{2,n}^d$ and 
conclude that there exist $D_1,D_2>1$ such that
$$\sum_n\sup_{f:\|f\|_\infty\le C}{\Pr}_f\left\{\sup_{t\in \mathbb{R}^d}|U_n(\pi_2(H_t))|>D_1\sqrt{h_{1,n}^dh_{2,n}^d}(\log n)/n\right\}\le
C_2\sum_n\exp\left(-D_2\log n\right)<\infty.$$
Since $(\log n)/[n\sqrt{h_{1,n}^dh_{2,n}^d}]<<((\log n)/n)^{4/(8+d)}$ we obtain
\begin{equation}\label{ust}
\sup_{t\in \mathbb{R}^d}\frac{1}{h_{1,n}^dh_{2,n}^d}|U_n(\pi_2(H_t))|=o_{\rm a.s.}(n^{-4/(8+d)})\ \ {\rm uniformly\ in}\ \ f\in{\cal P}_{C}.
\end{equation}

Having dealt with the first two terms in the last two lines of (\ref{hoef1}), we will now handle the third and last, namely,
\begin{equation}\label{TT}
T(t; h_{1,n}, h_{2,n})=\frac{1}{nh_{1,n}^dh_{2,n}^d}\sum_{i=1}^n(E_XH_t(X, X_i)-EH_t)
\end{equation}
or, setting, for ease of notation,
\begin{equation}\label{ge}
g(t,x)=E_XH_t(X, x)=E_{X}\left[
 L\left(\frac{t-X}{h_{2,n}}\alpha(f(X))\right)K\left(\frac{X-x}{h_{1,n}}\right)
(\alpha^d)'(f(X))\right],
\end{equation}
$$T(t;h_{1,n},h_{2,n})=\frac{1}{nh_{1,n}^dh_{2,n}^d}\sum_{i=1}^{n}(g(t,X_i)-Eg(t,X)).$$

Let $\cal G$ be the class of functions $\{g(t,\cdot):t\in\mathbb R^d\}$. We check that this class is of VC type and apply Talagrand's inequality once more. We have, for any $s, t\in\mathbb R^d$, and Borel probability measure $Q$,
\begin{eqnarray*}&&\!\!\!\!E_Q(g(t,x)-g(s,x))^2\\
&&\!\!\!\!\le \! \int \! E_X\!\left((\alpha^d)'(f(X))K\left(\frac{X-x}{h_{1,n}}\right)\right)^2 E_X\!\left(L\Big(\frac{t-X}{h_{2,n}}\alpha(f(X))\Big)
-L\Big(\frac{s-X}{h_{2,n}}\alpha(f(X))\Big)
\right)^2dQ(x)\\
&&\!\!\!\!\le \|(\alpha^d)'\|_\infty^2\|f\|_\infty  h_{1,n}^d\|K\|_2^2\int \left(L\Big(\frac{t-y}{h_{2,n}}\alpha(f(y))\Big)
-L\Big(\frac{s-y}{h_{2,n}}\alpha(f(y))\Big)
\right)^2f(y)dy\\
&&\!\!\!\!=\|(\alpha^d)'\|_\infty^2\|f\|_\infty h_{1,n}^d\|K\|_2^2E_f(\ell_t-\ell_s)^2
\end{eqnarray*}
where $\ell_{s}$ and $\ell_{t}$ are functions from the class
${\cal L}:=\left\{L\left(\frac{t- \cdot}{h}\alpha(f(\cdot))\right):t\in\mathbb{R}^d, h>0 \right\}$
which is VC for a constant envelope by Lemma \ref{vc1} (as $L$ satisfies the hypotheses of $K$ in that lemma -see Subsection 3.1-). This lemma then proves that for all $Q$ and for all $f\in{\cal P}_C$,
\begin{equation}\label{entg}
N({\cal G}, L_2(Q),\varepsilon)\le\left(\frac{R(c,p,d,K,C)h_{1,n}^{d/2}}{\varepsilon}\right)^{8d+20},\ \ 0<\varepsilon<R(c,p,d,K,C)h_{1,n}^{d/2}
\end{equation}
for some $R=R(c,p,d,K,C)$ depending only on the estipulated parameters, in particular, $\cal G$ is VC for the constant envelope $Rh_{1,n}^{d/2}$ (that depends on $n$), with characteristic constants $A=1$ and $v=8d+20$ independent of $n$ and $f\in{\cal P}_C$.

In order to apply Talagarand's (\ref{tal}) inequality, we need to estimate $Eg^2(t,X)$. With the change of variables
$x=t-h_{1,n}w-h_{1,n}z$, $y=t-h_{2,n}z$, $u=t-h_{1,n}w-h_{2,n}z-h_{1,n}s$, of determinant $h_{1,n}^{2d}h_{2,n}^d$, we obtain
\begin{eqnarray}\label{sig}
 &&Eg^2(t,X_1)\notag\\
&&=\int_{\mathbb R^d}\Big\{\int_{\mathbb R^d} f(x)K\Big(\frac{x-u}{h_{1,n}}\Big)
 L\left(\frac{t-x}{h_{2,n}}\alpha(f(x))\right)
(\alpha^d)'(f(x))dx\nonumber\\
&&~~~~\times \int_{\mathbb R^d} f(y)K\Big(\frac{y-u}{h_{1,n}}\Big)
 L\left(\frac{t-y}{h_{2,n}}\alpha(f(y))\right)
(\alpha^d)'(f(y))
dy\Big\}f(u)du\nonumber\\
&&\le ||f||_\infty^3h_{1,n}^{2d}h_{2,n}^d\|(\alpha^d)'\|_\infty^2\int_{\mathbb R^d}\int_{\mathbb R^d}\int_{\mathbb R^d}  L\Big(\big(\frac{h_{1,n}}{h_{2,n}}w+z\big) \alpha(f(t-h_{1,n}w-h_{2,n}z))\Big)\nonumber\\
&&~~~~~~~~~~\times L\Big(z \alpha(f(t-h_{2,n}z))\Big)K(s)K(s+w)dsdwdz\nonumber\\
&&\le B_K||f||_\infty^3h_{1,n}^{2d}h_{2,n}^d\|(\alpha^d)'\|_\infty^2 c^{-1},
\end{eqnarray}
where the last inequality follows from the bounded support of $L$ and $K$ and $\alpha(t)\ge c$. Then, since  the envelope $F$ of the VC class $\cal G$ can be taken to be $B_Kh_{1,n}^{d/2}$ and $\sigma^2$ to be $B_KC^3h_{1,n}^{2d}h_{2,n}^dc^{-1}$ (by (\ref{entg}) and (\ref{sig})), for constants $B_K$ that depend only on $K$, we get by (\ref{talc}) and (\ref{tal}) that there exist constants $D_1,D_2>1$ depending only on $K$, $C$, $d$, $p$ and $c$ such that 
$$\sup_{f\in{\cal P}_{C}}{\Pr}_f\left\{\left\|\sum_{i=1}^n(g(\cdot,X_i)-Eg(\cdot,X))\right\|_{\infty}\ge D_1\sqrt{nh_{1,n}^{2d}h_{2,n}^d\log n}\right\}\le C_2\exp\left(-D_2\log n\right),$$
and note that 
$$\sqrt{nh_{1,n}^{2d}h_{2,n}^d\log n}/(nh_{1,n}^dh_{2,n}^d)=((\log n)/n)^{4/(8+d)}$$
we get that
\begin{equation}\label{T}
\sup_{t\in\mathbb R^d}|T(t;h_{1,n}h_{2,n})|=O_{\rm a.s.}\left([(\log n)/n]^{4/(8+d)}\right)\ \ {\rm uniformly\ in}\  f\in{\cal P}_C.
\end{equation}

Combining the estimates (\ref{diag}), (\ref{ust}) and (\ref{T}) into (\ref{hoef1}) yields
\begin{equation}\label{varfinal}
\sup_{t\in\mathbb R^d}\frac{1}{nh_{2,n}^d}\left|\sum_{i=1}^{n}L\left(\frac{t-X_i}{h_{2,n}}\alpha(f(X_i))\right)
(\alpha^d)'( f(X_i)) D(X_i;h_{1,n})\right|=O_{\rm a.s.}\left([(\log n)/n]^{4/(8+d)}\right)
\end{equation}
uniformly  in $f\in{\cal P}_C$.

Plugging in the estimates (\ref{easyterms}), (\ref{easyterm2}), (\ref{bb}) and (\ref{varfinal}) into the decompositions (\ref{one})-(\ref{four}) and (\ref{eps1})-(\ref{e5}) of $\hat f(t;h_{1,n},h_{2,n})-\bar f(t;h_{2,n})$, yields:

\begin{proposition}\label{real-ideal}
Under Assumptions \ref{ass3}, for any $C<\infty$ the difference between the actual and the ideal estimators of a density $f$ satisfies
$$\sup_{t\in\mathbb R^d}|\hat f(t;h_{1,n},h_{2,n})-\bar f(t;h_{2,n})|=O_{\rm a.s.}\left(\left(\frac{log n}{n}\right)^{4/(8+d)}\right)\ {\rm uniformly\ in}\ f\in{\cal P}_{C,4}.$$
Moreover,
$$\sup_{t\in\mathbb R^d}|\hat f(t;h_{1,n},h_{2,n})-\bar f(t;h_{2,n})-T(t,h_{1,n}h_{2,n})|=o_{\rm a.s.}\left(\left(\frac{log n}{n}\right)^{4/(8+d)}\right)\ {\rm uniformly\ in}\ f\in{\cal P}_{C,4}.$$
\end{proposition}

\medskip
\noindent{\bf 3.3. End of the proof of Theorem \ref{main0}.} Proposition \ref{real-ideal}  together with the results in Section 2 for the bias (Corollary \ref{4}) and the variance (Proposition \ref {varid}) of the ideal estimator complete the proof of the asymptotic estimate (\ref{main1}) in Theorem \ref{main0}. To prove (\ref{main2}), we note that, by (\ref{classic1}) and (\ref{classic2}), $\|\hat f(t;h_{1,n})-f\|_\infty=O_{a.s.}(((\log n)/n)^{2/(4+d)})$ uniformly in $f\in{\cal P}_{C,2}$, that is, there exists $\lambda<\infty$ such that
\begin{equation}\label{prelim}
\lim_{k\to\infty}\sup_{f\in{\cal P}_{C,2}}\Pr_f\left\{\sup_{n\ge k}\left(\frac{n}{\log n}\right)^{2/(4+d)}\|\hat f(t_i;h_{1,n})-f\|_\infty>\lambda\right\}=0.
\end{equation}
Since $\|\hat f(\omega)-f\|_\infty\le \lambda((\log n)/n)^{2/(4+d)}$ implies $\hat {\cal D}_r^n(\omega)\subset {\cal D}_r$ as soon as $r> \lambda((\log n)/n)^{2/(4+d)}$, (\ref{main2}) follows immediately from (\ref{main1}) and (\ref{prelim}). This concludes the proof of Theorem \ref{main0}.

\section{Estimation of densities in $C^6(\mathbb R)$: Proof of Theorem \ref{h6main0}}

In this section we make the following assumptions on the kernel $K$, a new kernel $G$, the clipping function $p$, the densities $f$ and the bandsequences:

\medskip
\begin{assumptions}\label{h6ass3}
 We assume that the kernel $K$ is non-negative, bounded and is symmetric about zero, has support contained in $[-T,T]$, $T<\infty$, integrates to 1 and has  a uniformly bounded second derivative.  We also assume that the densities $f$ are bounded and have at least six bounded derivatives,
\begin{equation}
f\in{\cal P}_{C,6}:=\{f\ {\rm is\ a\ density}:\|f^{(k)}\|_\infty\le C, 0\le k\le 6\}
\end{equation}
for some $C<\infty$.  We assume that $G$ is a fourth order kernel $G$ supported by $[-T_G,T_G]$ for some $T_G<\infty$, integrates to $1$, is symmetric about zero and has two  bounded, continuous derivatives. The nondecreasing clipping function $p:\mathbb {R}\rightarrow \mathbb {R}$ is assumed to have a bounded and continuous derivative, $p(s)\ge 1$ for all $s$ and $p(s)=s$ for all $s\ge t_0\ge 1$, where $c$ and $t_0$ are fixed constants. We set $h_{1,n}=((\log n)/n)^{1/5}$, $h_{2,n}=h_{4,n}=((\log n)/n)^{1/13}$, $h_{3,n}=((\log n)/n)^{1/11}$, $n\in\mathbb N$.
 \end{assumptions}

The ideal estimator we study in this section is as defined in (\ref{h6idealest}), and the corresponding true estimator as in (\ref{h6realest}).

By (\ref{idealrateJ}), in order to prove Theorem \ref{h6main0}, it suffices to show that the uniform convergence rate of the discrepancy between the true estimator (\ref{h6realest}) and the ideal estimator (\ref{h6idealest}) is of the same order.

We will use $\delta(x), D(x;h_{1,n})$ and $b(x;h_{1,n})$ as defined in the last section, and note that  (\ref{delta}), (\ref{deltabd}), (\ref{classic1}), (\ref{classic2}) and (\ref{deltatalor}) still hold (with $d=1$).
Following the proof of Theorem 2.3 in Gin\'e and Guillou (2002) or Proposition 1 in Gin\'e and Sang (2010), under the conditions on $G$ and $f$ in Assumption \ref{h6ass3}, it is easy to show that
\[
\sup_{x\in \mathbb{R}}|f_{G_1}(x;h_{3,n})-Ef_{G_1}(x;h_{3,n})|=O_{\rm a.s.}\left(\sqrt{\frac{\log h_{3,n}^{-1}}{n h_{3,n}^3}}\right)
\]
and
\[
\sup_{x\in \mathbb{R}}|f_{G_2}(x;h_{4,n})-Ef_{G_2}(x;h_{4,n})|=O_{\rm a.s.}\left(\sqrt{\frac{\log h_{4,n}^{-1}}{n h_{4,n}^5}}\right)
\]
uniformly in $f\in{\cal P}_{C,6}$, and classical bias computations with $m$-th order kernels give
\[
\sup_{x\in \mathbb{R}}|Ef_{G_1}(x;h_{3,n})-f^\prime(x)|\le \frac{1}{24}\left(\int G(u)u^4du\right)\|f^{(5)}\|_\infty h_{3,n}^4,
\]
and
\[
\sup_{x\in \mathbb{R}}|Ef_{G_2}(x;h_{4,n})-f''(x)|\le \frac{1}{24}\left(\int G(u)u^4du\right)\|f^{(6)}\|_\infty h_{4,n}^4
\]
(recall the definitions of $f_{G_1}$ and $f_{G_2}$ from (\ref{first}) and (\ref{second})). Define 
$$\xi(x)=f_{G_1}(x;h_{3,n})-f'(x),\;\;  \zeta(x)=f_{G_2}(x;h_{4,n})-f''(x).$$ 
These estimates give
\begin{equation}\label{h6xi}
\sup_{x\in \mathbb{R}}|\xi(x)|=O_{\rm a.s.}(n^{-4/11}(\log n)^{4/11})\ \ {\rm uniformly\ in}\ f\ {\rm such\ that}\ f\in{\cal P}_{C,6}
\end{equation}
and
\begin{equation}\label{h6zeta}
\sup_{x\in \mathbb{R}}|\zeta(x)|=O_{\rm a.s.}(n^{-4/13}(\log n)^{4/13})\ \ {\rm uniformly\ in}\ f\ {\rm such\ that}\ f\in{\cal P}_{C,6},
\end{equation}
and, since $h_{1,n}$ is as in Section 3, (\ref{zero}) gives that 
\begin{equation}\label{h6zero}
\sup_{x\in \mathbb{R}}|\delta(x)|=O_{\rm a.s.}(n^{-2/5}(\log n)^{2/5})\ \ {\rm uniformly\ in}\ f\in{\cal P}_{C,2}.
\end{equation}

For $\beta$ and $\hat\beta(x;h_{1,n},h_{3,n},h_{4,n})$ as defined below (\ref{h6idealest}), define $\rho(x)$ so that
$$\frac{1}{1+h_{2,n}^2\hat\beta(x;h_{1,n},h_{3,n},h_{4,n})}=(1+\rho(x))\frac{1}{1+h_{2,n}^2\beta(x)}.$$
Then, with some elementary but tedious work,
\begin{eqnarray}
\rho(x)\!\!\!&=&\!\!\!\frac{h_{2,n}^2(\beta-\hat\beta)}{1+h_{2,n}^2\hat\beta}\nonumber\\
\!\!\!&=&\!\!\!\frac{h_{2,n}^2\{[(1+\delta(x))^6-1]S(x)-[f''(x)+\zeta(x)](D+b)-f(x)\zeta(x)+4f'(x)\xi(x)+2\xi(x)^2]\}
}{24\tau_2c^6p^3(c^{-2}\hat f(x;h_{1,n}))/\tau_4+h_{2,n}^2[f_{G_2}(x;h_{4,n})\hat f(x; h_{1,n})-2(f_{G_1}(x;h_{3,n}))^2]}\nonumber\\
\label{h6rhodef}
\end{eqnarray}
where $S(x):=f^{\prime\prime}(x)f(x)-2(f^\prime(x))^2$, $D=D(x;h_{1,n})$, $b=b(x;h_{1,n})$ and $\tau_k$ is the absolute $k$-th moment of $K$.
Now, since the denominator is bounded away from zero ($p$ is, and the second summand in the denominator tends to zero uniformly in $f$),  the bounds (\ref{h6xi})-(\ref{h6zero}) and (\ref{classic1}), (\ref{classic2}) give
\begin{equation}\label{h6rho}
\sup_{x\in \mathbb{R}}|\rho(x)|=O_{\rm a.s.}(n^{-6/13}(\log n)^{{6}/{13}})\ \ {\rm uniformly\ in}\ f\ {\rm such\ that}\ f\in{\cal P}_{C,6}.
\end{equation}
 The definition of $\rho$
allows us to write
\begin{equation}\label{gam}
\hat\gamma-\gamma=\gamma(\rho+\delta+\delta\rho).
\end{equation}
Recall the definitions (\ref{L1Lfunctions}) of  the functions $L$ and $L_1$ from last section (with $d=1$), the definition of $\gamma=\gamma_{h_{2,n}}$ from (\ref{h6idealest})  and that  of $\hat\gamma$ from (\ref{h6realest}), and note that $\gamma$ is bounded above and away from zero.
We then have
$$K\left(\frac{t-X_i}{h_{2,n}}\hat \gamma(X_i;h_{1,n},h_{2,n}, h_{3,n},h_{4,n})\right)
=K\left(\frac{t-X_i}{h_{2,n}}\gamma(X_i)(1+\delta(X_i)+\rho(X_i)+\delta(X_i)\rho(X_i))\right)$$
$$=K\left(\frac{t-X_i}{h_{2,n}}\gamma(X_i)\right)+K^\prime\left(\frac{t-X_i}{h_{2,n}}\gamma(X_i)\right)
\frac{t-X_i}{h_{2,n}}\gamma(X_i)(\delta(X_i)+\rho(X_i)+\delta(X_i)\rho(X_i))+\delta_2(t,X_i)$$
$$=K\left(\frac{t-X_i}{h_{2,n}}\gamma(X_i)\right)+
L_1\left(\frac{t-X_i}{h_{2,n}}\gamma(X_i)\right)(\delta(X_i)+\rho(X_i)+\delta(X_i)\rho(X_i))+\delta_2(t,X_i),
$$
where
\begin{equation}\label{h6delta2}
\delta_2(t,X_i)=\frac{K^{\prime\prime}(\Xi)}{2}
\frac{(t-X_i)^2}{h_{2,n}^2}\gamma^2(X_i)(\delta(X_i)+\rho(X_i)+\delta(X_i)\rho(X_i))^2,
\end{equation}
$\Xi$  being a (random) number between $\frac{t-X_i}{h_{2,n}}\gamma(X_i)$
and $\frac{t-X_i}{h_{2,n}}\gamma(X_i)(1+\delta(X_i)+\rho(X_i)+\delta(X_i)\rho(X_i)).$
Then, plugging this development and the development (\ref{gam}) of $\hat\gamma$ in the definition (\ref{h6realest}) of $\hat f$, we obtain the following, where we drop all the arguments for brevity,
\begin{eqnarray}
\hat f(t;h_{1,n},h_{2,n},h_{3,n},h_{4,n})&&\!\!\!\!\!\!\!\!\!\!\!-\bar f(t;h_{2,n})=\frac{1}{n h_{2,n}}\sum_{i=1}^{n}L\delta\gamma+\frac{1}{n h_{2,n}}\sum_{i=1}^{n}L\rho\gamma\label{h6mainterms}\\
&&\!\!\!\!\!\!+\frac{1}{n h_{2,n}}\sum_{i=1}^{n}[L\delta\rho\gamma+L_1(\delta +\rho+\delta\rho)^2\gamma+(1+\delta +\rho+\delta\rho)\gamma\delta_2].\label{remainder}
\end{eqnarray}
First, we check the order of the term (\ref{remainder}). Since $\gamma$ is bounded above and below, $\delta\rightarrow 0$  a.s. by (\ref{h6zero}) and $\rho\rightarrow 0$ a.s. by (\ref{h6rho}), we have $\left|\frac{t-X_i}{h_{2,n}}\right|\le B_1|\Xi|$ for some constant $B_1<\infty$ and therefore $\frac{(t-X_i)^2}{h_{2,n}^2}K''(\Xi)$ is bounded since $K''$ has bounded support. This together with (\ref{h6zero}) and (\ref{h6rho}) (which show $\rho<<\delta$) and the definition of $\delta_2$ in (\ref{h6delta2}), imply that $\delta_2(t, X_i)\le B_2\|f\|_\infty \delta^2(X_i)$ for some constant $B_2$. Then, again by (\ref{h6zero}) and (\ref{h6rho}), (\ref{remainder}) is dominated by $\frac{1}{nh_{2,n}}\sum_{i=1}^n\delta^2\gamma$ which has order $n^{-47/65}({\log n})^{47/65}$. Therefore,
\begin{equation}\label{94}
(\ref{remainder})=o_{\rm a.s.}(n^{-6/13}({\log n})^{6/13})
\end{equation}
uniformly in $t\in\mathbb R$ and in $f\in{\cal P}_{C,6}$.
Next, we will check the order of the two terms at the right in (\ref{h6mainterms}). Each of them will require further decompositions. For the first term,
using the decomposition  (\ref{deltatalor}) of $\delta$, we have, just as in (\ref{eps1})-(\ref{e5}),
\begin{eqnarray}
&&\frac{1}{n h_{2,n}}\sum_{i=1}^{n}L\delta\gamma=\frac{1}{n h_{2,n}}\sum_{i=1}^{n}L\left(\frac{t-X_i}{h_{2,n}}\gamma(X_i)\right)\delta(X_i)\gamma(X_i)\nonumber\\
&&=\frac{1}{n h_{2,n}}\sum_{i=1}^{n}L\left(\frac{t-X_i}{h_{2,n}}\gamma(X_i)\right)
\frac{\alpha'(f(X_i))b(X_i;h_{1,n})
}{\alpha(f(X_i))}\gamma(X_i)\label{h6b}\\
&&~~+\frac{1}{n h_{2,n}}\sum_{i=1}^{n}L\left(\frac{t-X_i}{h_{2,n}}\gamma(X_i)\right)
\frac{\alpha'(f(X_i))D(X_i;h_{1,n})
}{\alpha(f(X_i))}\gamma(X_i)\label{h6D}\\
&&~~+\frac{1}{n h_{2,n}}\sum_{i=1}^nL\left(\frac{t-X_i}{h_{2,n}}\gamma(X_i)\right)\frac{\alpha''(\eta(X_i))}{\alpha(f(X_i))}[\hat f(X_i;h_{1,n})-f(X_i)]^2\gamma(X_i)\label{h6deltaequare}
\end{eqnarray}
Since the functions $L(x), \alpha''(\eta(x))$ and $\gamma(x)$ are bounded and the clipping function $p(x)$ is bounded away from zero, it can be easily seen, using (\ref{classic1}), (\ref{classic2}), that (\ref{h6deltaequare}) is dominated by
\begin{equation}\label{6delta2}
|(\ref{h6deltaequare})|=O_{\rm a.s.}\left(\frac{h_{1,n}^4}{h_{2,n}}\right)=O_{\rm a.s.}\left(\left(\frac{\log n}{n}\right)^{47/65}\right)\ \ {\rm uniformly\ on}\ \mathbb{R}\ {\rm and\ in}\ f\in{\cal P}_{C,2}.
\end{equation}
 As with the bounds for $C^4(\mathbb R)$ densities, the main terms in the present decomposition are (\ref{h6b})  and  (\ref{h6D}),  and they can be handled as in the previous section, basically using the Talagrand and Major inequalities: the bounds obtained have the same expressions as those bounds in terms of the bandwidths, which now of course are different. For the bias term, the bound is of the order of $h_{1,n}^2h_{2,n}^2=(\log n)/n)^{36/65}$, see (\ref{bb}). For the variance term, it is of the order of $(\log n)/n)^{1/2}h_{2,n}^{-1/2}=(n/\log n)^{-6/13}$ (see (\ref{diag}), (\ref{ust}) and particularly, (\ref{T})). So, we have, uniformly in $t\in\mathbb R$ and in $f\in{\cal P}_{C,2}$,
\begin{equation}\label{bv}
|(\ref{h6b})|=o_{\rm a.s.}\left(\left(\frac{\log n}{n}\right)^{6/13}\right),\;\;\ |(\ref{h6D})| =O_{\rm a.s.}\left(\left(\frac{\log n}{n}\right)^{6/13}\right).
\end{equation}

Finally, we check the order of the second term in (\ref{h6mainterms}), $\frac{1}{n h_{2,n}}\sum_{i=1}^{n}L\rho\gamma$. We give the details for the estimation of this term because it is different from the previous case.
In the definition (\ref{h6rhodef}), by assumption \ref{h6ass3}, (\ref{classic1}), (\ref{classic2}), (\ref{h6xi}) and (\ref{h6zeta}),
$\sup_{x\in \mathbb{R}}|f_{G_2}(x;h_{4,n})\hat f(x; h_{1,n})-2(f_{G_1}(x; h_{3,n}))^2|$ is bounded almost surely.
Hence, since  $p\ge 1$, we have, for the denominator $de(\rho)$ of $\rho$,
\begin{eqnarray}
&&\inf_{x\in \mathbb{R}}|de(\rho)|\nonumber\\
&&:=\inf_{x\in \mathbb{R}}|24\tau_2c^6p^3(c^{-2}\hat f(x;h_{1,n}))/\tau_4+h_{2,n}^2f_{G_2}(x;h_{4,n})\hat f(x; h_{1,n})-2h_{2,n}^2(f_{G_1}(x; h_{3,n}))^2|\nonumber\\
&&\ge B_d\nonumber
\end{eqnarray}
almost surely, for some universal constant $B_d>0$ if $n$ is large enough. Therefore, almost surely,
\begin{eqnarray}
&&\!\!\!\!\!\sup_{t\in \mathbb{R}}\left|\frac{1}{n h_{2,n}}\sum_{i=1}^{n}L\rho\gamma\right|\nonumber\\
&&\!\!\!\!\!\le\sup_{t\in \mathbb{R}}\frac{h_{2,n}}{B_d n}\sum_{i=1}^n \left|L\left(\frac{t-X_i}{h_{2,n}}\gamma(X_i)\right)\gamma(X_i)[(1+\delta(X_i))^6-1]S(X_i)\right|\label{h6deltafunc}\\
&&\!\!\!\!\!+\sup_{t\in \mathbb{R}}\frac{h_{2,n}}{B_d n}\sum_{i=1}^n\left|L\left(\frac{t-X_i}{h_{2,n}}\gamma(X_i)\right)\gamma(X_i)[f''(X_i)+\zeta(X_i)]
[D(X_i;h_{1,n})+b(X_i;h_{1,n})]\right|\label{h6D+b}\\
&&\!\!\!\!\!+\sup_{t\in \mathbb{R}}\frac{h_{2,n}}{B_d n}\sum_{i=1}^n\left|L\left(\frac{t-X_i}{h_{2,n}}\gamma(X_i)\right)\gamma(X_i)f(X_i)\zeta(X_i)\right|\label{h6zetaterm}\\
&&\!\!\!\!\!+\sup_{t\in \mathbb{R}}\frac{h_{2,n}}{B_d n}\sum_{i=1}^n \left|4L\left(\frac{t-X_i}{h_{2,n}}\gamma(X_i)\right)\gamma(X_i)f'(X_i)\xi(X_i)\right|\label{h6xiterm}\\
&&\!\!\!\!\!+\sup_{t\in \mathbb{R}}\frac{h_{2,n}}{B_d n}\sum_{i=1}^n\left |2L\left(\frac{t-X_i}{h_{2,n}}\gamma(X_i)\right)\gamma(X_i)\xi(X_i)^2\right| \label{h6xisquare}
\end{eqnarray}

Since the functions $L(x), ~\gamma(x),~f(x), f'(x)$ and $f''(x)$ are bounded,  it follows that (\ref{h6deltafunc}) and (\ref{h6D+b}) are of the order $O_{\rm a.s.}([(\log n)/n]^{31/65})=o_{\rm a.s.}([(\log n)/n]^{6/13})$ uniformly in $t\in \mathbb R$ and $f\in{\cal P}_{C,2}$, the first by the estimate (\ref{h6zero}) and the second by the classical bounds (\ref{classic1}) and (\ref{classic2}). It also follows from  (\ref{h6xi}), that (\ref{h6xisquare}) is    $O_{\rm a.s.}([(\log n)/n]^{8/11+1/13})$ uniformly in $t$ and $f$.
Now we estimate the term (\ref{h6zetaterm}).
If a class of functions is VC type, so is the class of its absolute values (covering numbers are smaller), hence, Lemma \ref{vc1} shows that the classes of functions
\begin{equation}\label{M}
{\cal N}_n:=\left\{N(x)=\left|L\left(\frac{t-x}{h_{2,n}}\gamma(x)\right)\gamma(x)f(x)\right|:t\in \mathbb{R}\right\}
\end{equation}
are of VC type  for envelopes of the order of  $||f||_\infty^{3/2} O(1)$ and admitting the same characteristic constants $A$ an $v$.   If we set
 $$N_i(t)=\left|L\left(\frac{t-X_i}{h_{2,n}}\gamma(X_i)\right)\gamma(X_i)f(X_i)\right|,$$
it then follows by the properties of $L$ and $p$, that
$$\sup_{t\in \mathbb{R}}EN_i(t)\lessim ||f||_\infty^{3/2}h_{2,n},\;\;
\sup_{t\in \mathbb{R}}EN_i^2(t)\lessim ||f||_\infty^3h_{2,n},\;\;
\sup_{t\in \mathbb{R}}N_i(t)\lessim ||f||_\infty^{3/2},$$
So, we have
\begin{eqnarray*}
&&\!\!\!\sup_{t\in \mathbb{R}}\frac{h_{2,n}}{B_d n}\sum_{i=1}^n\left|L\left(\frac{t-X_i}{h_{2,n}}\gamma(X_i)\right)\gamma(X_i)f(X_i)\right|\\
&&~~\le \sup_{t\in \mathbb{R}}\left|\frac{h_{2,n}}{B_dn}\sum_{i=1}^{n} [N_i(t)-EN_i(t)]\right|+ \sup_{t\in \mathbb{R}}h_{2,n}|EN_1(t)|/B_d\\
&&~~\lessim   \sup_{t\in \mathbb{R}}\left|\frac{h_{2,n}}{n}\sum_{i=1}^{n} [N_i(t)-EN_i(t)] \right|+||f||^{3/2}n^{-2/13}(\log n)^{2/13},
\end{eqnarray*}
and Talagrand's inequality gives that there exist $D_1,D_2>1$ such that 
$$\sum_n\sup_{f\in{\cal P}_{C,2}}{\Pr}_f\left\{\sup_{t\in \mathbb{R}}\left|\sum_{i=1}^{n} [N_i(t)-EN_i(t)] \right|\ge D_1\sqrt{nh_{2,n}\log n}\right\}\le
C_2\sum_n\exp\left(-D_2\log n\right)<\infty.$$
The last two estimates yield
$$\sup_{t\in \mathbb{R}}\frac{h_{2,n}}{B_d n}\sum_{i=1}^n\left|L\left(\frac{t-X_i}{h_{2,n}}\gamma(X_i)\right)\gamma(X_i)f(X_i)\right|=O_{\rm a.s.}(n^{-2/13}(\log n)^{2/13})\ \ {\rm uniformly\ in}\ \ f\in{\cal P}_{C,2}.$$
Combining this with the bound (\ref{h6zeta}) for $\zeta$, gives
\begin{eqnarray}
&&\sup_{t\in \mathbb{R}}\frac{h_{2,n}}{B_d n}\sum_{i=1}^n\left|L\left(\frac{t-X_i}{h_{2,n}}\gamma(X_i)\right)\gamma(X_i)f(X_i)\zeta(X_i)\right|\nonumber\\
&&\le \sup_{x\in \mathbb{R}}|\zeta(x)|\times \sup_{t\in \mathbb{R}}\frac{h_{2,n}}{B_d n}\sum_{i=1}^n\left|L\left(\frac{t-X_i}{h_{2,n}}\gamma(X_i)\right)\gamma(X_i)f(X_i)\right| \nonumber\\
&&=O_{\rm a.s.}([(\log n)/n]^{6/13})\;\;\;a.s. \nonumber
\end{eqnarray}
Note that $h_{2,n}^2\backsimeq n^{-2/13}(\log n)^{2/13}$ and $h_{4,n}\backsimeq n^{-1/13}(\log n)^{1/13}$ play  critical roles in this estimation.
The same argument produces the same bound for the term (\ref{h6xiterm}).

So, we have shown that each of the three terms in the decomposition (\ref{h6mainterms}) and (\ref{remainder}) of
$\hat f(t;h_{1,n},h_{2,n},h_{3,n}, h_{4,n})-\bar f(t;h_{2,n})$ is at most of the order $O_{\rm a.s.}([(\log n)/n]^{6/13})$ uniformly in $t$ and $f$, that is,

\begin{proposition}\label{h6real-ideal}
Under the Assumptions \ref{h6ass3}, for any $C<\infty$ we have:
$$\sup_{t\in \mathbb R}\left|\hat f(t;h_{1,n},h_{2,n}, h_{3,n}, h_{4,n})-f_{JKH}(t;h_{2,n})\right|=O_{\rm a.s.}\left(\left(\frac{\log n}{n}\right)^{6/13}\right)\ \ {\rm uniformly\ in}\ \ f\in{\cal P}_{C,2}.$$
\end{proposition}

Combining this proposition with the estimates of the bias and variance terms of the ideal estimator, respectively Corollary \ref{6} and Proposition \ref{varid}, yields Theorem \ref{h6main0}.


\medskip

\noindent E. Gin\'e\hfill\break
\noindent Department of Mathematics, U-3009\hfill\break
\noindent University of Connecticut\hfill\break
\noindent Storrs, CT 06269\hfill\break
\noindent gine@math.uconn.edu

\medskip
\noindent H. Sang\hfill\break
\noindent National Institute of Statistical Sciences\hfill\break
\noindent PO Box 14006\hfill\break
\noindent Research Triangle Park, NC 27709 \hfill\break
\noindent sang@niss.org
\end{document}